\documentclass[reqno]{amsart}
\usepackage{textcomp}

\usepackage[dvipsnames]{xcolor}
\usepackage{amsmath}
\usepackage{amsthm}
\usepackage{hyperref}
\usepackage{amsfonts,graphics,amsthm,amsfonts,amscd,latexsym}
\usepackage{tikz-cd}
\usepackage{epsfig}
\usepackage{flafter}
\usepackage{longtable}
\usepackage{mathtools}
\usepackage{comment}
\usepackage{stmaryrd}

\usepackage{mathabx,epsfig}

\hypersetup{
    colorlinks=true,    
    linkcolor=blue,          
    citecolor=blue,      
    filecolor=blue,      
    urlcolor=blue           
}
\usepackage{tikz}
\usetikzlibrary{graphs,positioning,arrows,shapes.misc,decorations.pathmorphing}

\tikzset{
    >=stealth,
    every picture/.style={thick},
    graphs/every graph/.style={empty nodes},
}

\tikzstyle{vertex}=[
    draw,
    circle,
    fill=black,
    inner sep=1pt,
    minimum width=5pt,
]
\usepackage[position=top]{subfig}
\usepackage{amssymb}
\usepackage{color}

\calclayout

\usetikzlibrary{decorations.pathmorphing}
\tikzstyle{printersafe}=[decoration={snake,amplitude=0pt}]

\newcommand{\supp}{\operatorname{supp}}

\newcommand{\pp}{\mathbb{P}}

\newcommand{\qq}{\mathbb{Q}}

\newcommand{\cc}{\mathbb{C}}

\def\O#1.{\mathcal {O}_{#1}}			
\def\pr #1.{\mathbb P^{#1}}				
\def\af #1.{\mathbb A^{#1}}			
\def\ses#1.#2.#3.{0\to #1\to #2\to #3 \to 0}	
\def\xrar#1.{\xrightarrow{#1}}			
\def\K#1.{K_{#1}}						
\def\bA#1.{\mathbf{A}_{#1}}			
\def\bM#1.{\mathbf{M}_{#1}}				
\def\bL#1.{\mathbf{L}_{#1}}				
\def\bB#1.{\mathbf{B}_{#1}}				
\def\bK#1.{\mathbf{K}_{#1}}			
\def\subs#1.{_{#1}}					
\def\sups#1.{^{#1}}

\usepackage{tikz}
\usetikzlibrary{matrix,arrows,decorations.pathmorphing}

  \newtheorem{theorem}{Theorem}[section]
  \newtheorem{lemma}[theorem]{Lemma}
  \newtheorem{proposition}[theorem]{Proposition}
  \newtheorem{corollary}[theorem]{Corollary}

  \newtheorem{notation}[theorem]{Notation}
  \newtheorem{definition}[theorem]{Definition}
  \newtheorem{example}[theorem]{Example}

  \newtheorem{question}[theorem]{Question}

\newtheorem{remark}[theorem]{Remark}

\theoremstyle{remark}

\numberwithin{equation}{section}

\usepackage[all]{xy}

\begin{document}

\thanks{E.A.S. was supported by ERC StG Saphidir No. 101076412 and Brazil-France Network in Mathematics. J.M. was partially supported by NSF research grant DMS-2443425.}

\title[Algebraic tori in the complement of quartic surfaces]{
Algebraic tori in the complement of quartic surfaces}

\author[E. Alves da Silva]{Eduardo Alves da Silva}
\address{Departement Mathematik und Informatik,  Universität Basel, Spiegelgasse 1, 4051 Basel, Switzerland.
}
\email{eduardo.alvesdasilva@unibas.ch}

\author[F.~Figueroa]{Fernando Figueroa}
\address{Department of Mathematics, Northwestern University, Evanston, Il 60208, USA
}
\email{fernando.figueroa@northwestern.edu}

\author[J.~Moraga]{Joaqu\'in Moraga}
\address{UCLA Mathematics Department, Box 951555, Los Angeles, CA 90095-1555, USA
}
\email{jmoraga@math.ucla.edu}


\begin{abstract}
Let $B\subset \pp^3$ be an slc quartic surface.
The existence of an embedding $\mathbb{G}_m^3\hookrightarrow \pp^3\setminus B$
implies that $B$ has coregularity zero.
In this article, we initiate the classification of coregularity zero slc quartic surfaces $B\subset \pp^3$ for which $\pp^3\setminus B$ contains an algebraic torus $\mathbb{G}_m^3$, and equivalently, the classification of cluster type pairs
$(\pp^3,B)$.
Along the way, we give criteria
for a log Calabi--Yau pair $(X,B)$ over a toric variety $T$
to be of cluster type. 
\end{abstract} 

\maketitle

\setcounter{tocdepth}{1} 
\tableofcontents

\section{Introduction}

Calabi--Yau varieties are one of the three building blocks of algebraic varieties.
Elliptic curves and 
quartic surfaces are among the simplest Calabi--Yau varieties.
Smooth quartic surfaces are the well-known K3 surfaces
that are extensively studied in the literature.
Quartic surfaces with isolated nodal points
and nodal curves are classic topics as well.
The study of singular quartic surfaces
dates back to Kummer~\cite{Kum65}.
These surfaces were considered by many mathematicians in the 19th century including Clebsch~\cite{Cle69},
Klein~\cite{Kle74,Kle85}, 
Noether~\cite{Noe70},
Rohn~\cite{Roh84}, and
Segre~\cite{Seg84}.
In~\cite{Jes16}, Jessop gives a complete account of the leading properties of quartic surfaces with isolated singularities or nodal curves. 
In the 20th century, singular quartic hypersurfaces
have been studied by Kato and Naruki~\cite{KN82}, 
Umezu~\cite{Ume84}, Takahashi, Watanabe, and Higuchi~\cite{TWH82}, and Urabe~\cite{Ura84}.
A modern classification of non-normal quartic surfaces was achieved by Urabe in~\cite{Ura86}.
In~\cite{Mel20}, Mella proved that a rational quartic surface $S$ is Cremona equivalent to a plane, i.e., 
it can be transformed into a plane by performing a birational automorphism of the ambient projective space $\pp^3$.
The singularities of $S$ in Mella's result are necessary.
Indeed, in~\cite{Ogu17}, Oguiso has produced examples of isomorphic smooth quartic surfaces
that are not Cremona equivalent.
In this article, we study log Calabi--Yau pairs $(\pp^3,B)$
where $B$ is a quartic hypersurface, from the perspective of birational geometry.
This approach was initiated by Ducat in~\cite{Duc24}.
In that article, Ducat proved that a log Calabi--Yau pair $(\pp^3,B)$ of coregularity zero and index one is {\em log rational}, i.e., there exists a crepant birational map
\begin{equation}\label{eq:cbir} 
\phi\colon (\pp^3,H_0+H_1+H_2+H_3) \dashrightarrow (\pp^3,B).
\end{equation} 
Roughly speaking, a crepant birational map is a birational map that respects the boundary divisor (see Definition~\ref{def:cbir}). This says that for such pairs there is a unique crepant birational equivalence class.
In the previous statement, the index one condition simply implies that $B$ is a reduced quartic surface, while
the coregularity condition measures singularities of pairs (see Definition~\ref{def:coreg}). For such pairs, one has coregularity $2$ if and only $B$ is canonical, and coregularity less or equal than one if and only if $B$ is strictly semi-log canonical. In the same work Ducat showed that for pairs $(\mathbb{P}^2,B)$ of coregularity 1 the crepant birational classes can be parametrized by the moduli space of elliptic curves, while the case of coregularity 2 is much more involved.
The work of Ducat heavily relied on constructions coming from Cremona transformations~\cite{Kat87,DH16}. 
In a similar vein, Loginov, Vasilkov, and the third author proved that a general smooth rational Fano $3$-fold $X$ admits a boundary $B$ for which $(X,B)$ is log rational (see, e.g.,~\cite{LMV24}).

A three-dimensional {\em cluster type} pair is a log rational pair $(X,B)$ for which the birational map in~\eqref{eq:cbir} does not contract divisors intersecting $\mathbb{G}_m^3$ (see Definition~\ref{def:ct} for the general definition).
Cluster type pairs and cluster type varieties\footnote{A cluster type variety $X$ is a variety that admits a cluster type log Calabi--Yau pair $(X,B)$.} are closely related to cluster varieties. However, cluster type pairs give a more general class of objects.
For instance, every del Pezzo surface $X$ of degree at least two is a cluster type variety (see~\cite[Theorem 2.1]{ALP23} and~\cite[Lemma 1.13]{GHK15}). 
On the other hand, there are examples of del Pezzo surfaces of degree one which are not cluster type.
However, there is a dense open set of del Pezzo surfaces of degree one that are cluster type.
The property of being cluster type is constructible in families of smooth Fano varieties (see~\cite{JM24}).
Due to~\cite{Duc24}, the three-dimensional projective space $\pp^3$ admits a plethora of log Calabi--Yau pairs $(\pp^3,B)$ that are log rational.
In this article, we initiate the classification of {\em cluster type} log Calabi--Yau pairs of the form $(\pp^3,B)$.
Our main theorem is the following.

\begin{theorem}\label{introthm:cluster-type-(P^3,B)}
Let $(\pp^3,B)$ be a log Calabi--Yau pair
of index one and coregularity zero.
Assume that $B$ is non-normal.
If the nodal locus of $B$ is not contained in a plane,
then $(\pp^3,B)$ is cluster type.
\end{theorem}

In the case that $B$ is not irreducible, we can completely determine if the pair $(\pp^3,B)$ is of cluster type.

\begin{theorem}\label{thm:red}
Let $(\pp^3,B)$ be a log Calabi--Yau pair of index one, coregularity zero, and $B$ reducible.
Then, the pair $(\pp^3,B)$ is of cluster type if and only if the $B$ is not the sum of a hyperplane and a cubic surface intersecting smoothly along a nodal plane cubic.
\end{theorem} 

Thus, we can decide whether a log Calabi--Yau pair $(\pp^3,B)$ is of cluster type unless $B$ is an irreducible surface whose non-divisorial
log canonical centers are contained in a plane. 
Further, if $B$ has a point of multiplicity at least three, then in Proposition~\ref{prop:mult3} we show that the pair $(\pp^3,B)$ is of cluster type. 
In Example~\ref{ex:nodal-plane-conic} and Example~\ref{ex:nodal-line}, we provide quartic surfaces with nodal points along plane curves for which every point has multiplicity at most two.
These examples show that Theorem~\ref{introthm:cluster-type-(P^3,B)} does not settle the classification
of cluster type pairs of the form $(\pp^3,B)$.
However, the classification is complete for $B$ reducible due to Theorem~\ref{thm:red}. In Example \ref{ex:irred-normal}, we provide an irreducible normal quartic surface $B$ such that $(\pp^3,B)$ has coregularity zero and it is of cluster type.

Cluster type pairs were introduced by Corti in \cite{Cor25}; and Enwright and the last two authors in~\cite{EFM24}. 
In \cite[Theorem 16.2.3]{Cor25} and \cite[Theorem 1.3(3)]{EFM24}, they prove that whenever $X$ is a $\qq$-factorial Fano variety a log Calabi--Yau pair $(X,B)$ is of cluster type if and only if $X\setminus B$ is divisorially covered by algebraic tori.
In particular, Theorem~\ref{introthm:cluster-type-(P^3,B)} gives a partial classification of quartic surface $B\subset \pp^3$ for which there is an embedding $\mathbb{G}_m^3\subseteq \pp^3\setminus B$.

One of the main tools in order to prove Theorem~\ref{introthm:cluster-type-(P^3,B)} is a criteria to decide whether a log Calabi--Yau pair $(X,B)$ admitting a fibration to a toric variety is of cluster type. To do so, in Definition~\ref{def:rel-ct}, we introduce the notion of relative cluster type pairs. 
The following result states that under some mild assumptions
a $3$-dimensional log Calabi--Yau pair is of cluster type
if and only if it is of cluster type over $\pp^1$.

\begin{theorem}\label{introthm:relative-cluster-type}
Let $(X,B)$ be a log Calabi--Yau $3$-fold of index one and coregularity zero.
Let $\pi\colon (X,B)\rightarrow (\pp^1,\{0\}+\{\infty\})$
be a crepant fibration, i.e. the pair structure is the one obtained from the canonical bundle formula. 
Assume that $\pi^{-1}(0)\cup \pi^{-1}(\infty)\subseteq \supp(B)$.
Then, the pair $(X,B)$ is of cluster type if and only if
it is of cluster type over $\pp^1$.
\end{theorem}

In Theorem~\ref{thm:relative-cluster-type}, we give a higher dimensional version of the previous theorem, i.e., a version in which we replace $\pp^1$ with a possibly higher dimensional projective toric variety.
However, the previous version of the theorem is often the most useful one. Indeed, using pencils, it is frequently easy to construct fibrations to $\pp^1$ from higher birational models of a given variety.
Thus, Theorem~\ref{introthm:relative-cluster-type} is especially practical to understand whether a log Calabi--Yau pair $(\pp^3,B)$, with $B$ reducible, is of cluster type. In the case that $\pi\colon X\rightarrow \pp^1$ is a Mori fiber space, we get stronger criteria for the log Calabi--Yau pair $(X,B)$ to be of cluster type.

\begin{theorem}\label{introthm:relative-criteria-irreducible}
Let $(X,B)$ be a $\qq$-factorial log Calabi--Yau $3$-fold of index one and coregularity zero.
Let $\pi\colon (X,B)\rightarrow (\pp^1,\{0\}+\{\infty\})$ be a crepant fibration of relative Picard rank one.
Assume that $\pi^{-1}(0)\cup \pi^{-1}(\infty)\subseteq \supp(B)$.
Let $(F,B_F)$ be a general log fiber of $\pi$
and assume that $B_F$ is an irreducible nodal curve. 
Then, the following two conditions are equivalent:
\begin{enumerate}
\item the pair $(X,B)$ is of cluster type, and 
\item there exists a toric weighted blow-up $\widetilde{F}\rightarrow F$ extracting a unique log canonical place of $(F,B_F)$ over a nodal point of $B_F$ such that the strict transform $\widetilde{B}_F$ of $B_F$ in $\widetilde{F}$ has positive self-intersection.
\end{enumerate} 
\end{theorem} 

One of the interesting features of the previous theorem
is that it reduces a problem about three-dimensional birational geometry of projective varieties
to a problem about germs of two-dimensional singularities.
Indeed, once we know the self-intersection $B_F^2$, then Theorem~\ref{introthm:relative-criteria-irreducible} reduces the problem of understanding whether $(X,B)$ is of cluster type to study the toric surface singularity at the nodal point of the irreducible curve $B_F$.

The paper is organized as follows:
In Section~\ref{sec:prelim}, we explain some preliminary results about the index and coregularity of log Calabi--Yau pairs. Further, we recall the concepts of log rational and cluster type pairs. 
In Section~\ref{sec:rel-ct}, we prove Theorem~\ref{introthm:relative-cluster-type} regarding relative cluster type pairs.
In Section~\ref{sec:criteria-toric-blow-up}, we prove Theorem~\ref{introthm:relative-criteria-irreducible} that gives a criterion for certain log Calabi--Yau pairs of dimension three to be of cluster type in terms of toric blow-ups.
In Section~\ref{sec:reducible}, we use the previous two sections to classify cluster type pairs $(\pp^3,B)$ with $B$ reducible.
In Section~\ref{sec:irred}, we study whether a log Calabi--Yau pair $(\pp^3,B)$ with $B$ irreducible, is of cluster type.
Finally, in Section~\ref{sec:ex-and-quest}, we provide some examples and propose some questions for further research.

\subsection*{Acknowledgements}

This project was initiated by the authors at the \href{https://impa.br/en_US/eventos-do-impa/2024-2/v-latin-american-school-of-algebraic-geometry-and-applications-v-elga/}{V Latin American School of Algebraic Geometry}. The authors would like to thank the organizers for this collaborative space. 
Parts of this work were carried out while the first and third authors participated in the conference \href{https://aimath.org/pastworkshops/higherdimlogcy.html}{Higher-dimensional log Calabi--Yau pairs} at the \href{https://aimath.org/}{American Institute of Mathematics} (AIM). The authors would like to thank the organizers and staff at AIM for their working environment and hospitality.
The authors would like to thank Paul Hacking, Radu Laza, and Jose Ignacio Y\'a\~nez for very useful comments.

\section{Preliminaries}
\label{sec:prelim}

We work over the field of complex numbers $\cc$.
We write $\Sigma^n$ for the sum of the coordinate hyperplanes in $\pp^n$. In this section, we introduce some preliminary results
regarding log rational pairs and cluster type pairs.
In Subsection~\ref{subsec:2div} and Subsection~\ref{subsec:comp1}, we give two criteria for 
log Calabi--Yau $3$-folds to be of cluster type.
For the notion of log pairs, we refer the reader to~\cite{KM98}. 
For the singularities of the MMP, we refer the reader to~\cite{Kol13}.

\subsection{Cluster type pairs}\label{subsec:ct}
In this subsection, we recall the notions
of log rationality and cluster type pairs.
First, we recall the definition
of crepant birational equivalence.

\begin{definition}\label{def:cbir}
{\em 
Let $(X,B)$ and $(X',B')$ be two sub-pairs.
We say that they are {\em crepant birational equivalent}, 
denoted by $(X,B)\simeq_{\rm cbir}(X',B')$ if the following conditions are satisfied:
\begin{enumerate}
\item there is a common resolution
$p\colon Y\rightarrow X$ and $q\colon Y\rightarrow X'$, and 
\item we have $p^*(K_X+B)=q^*(K_{X'}+B')$.
\end{enumerate}
In the previous setting, we say that $q\circ p^{-1}\colon (X,B)\dashrightarrow (X',B')$ is a {\em crepant birational map}.
}
\end{definition}

\begin{notation}
{\em 
Let $p\colon Y\dashrightarrow X$ be a birational map.
Let $(X,B)$ be a sub-pair structure on $X$.
The {\em log pull-back} of $(X,B)$ in $Y$ 
is the unique divisor $B_Y$ on $Y$ for which
the birational map
$p\colon (Y,B_Y)\dashrightarrow (X,B)$ is a crepant birational map.
}
\end{notation}

\begin{definition}
{\em 
We say that $(X,B)$ is a {\em log Calabi--Yau pair} if
$(X,B)$ has log canonical singularities and $K_X+B\equiv 0$.
}
\end{definition}

If $(X,B)$ and $(X',B')$ are two crepant birational equivalent pairs, then it follows from the negativity lemma that $(X,B)$ is log Calabi--Yau
if and only if $(X',B')$ is log Calabi--Yau.

If $X$ is projective and $(X,B)$ is a log Calabi--Yau pair, then it follows from the work of Fujino and Gongyo that $m(K_X+B)\sim 0$ for a suitable positive integer $m$ (see, e.g.,~\cite[Theorem 1.2]{Gon13}).

\begin{definition}
{\em 
Let $(X,B)$ be a log Calabi--Yau pair.
The {\em index} of $(X,B)$ is the smallest positive integer $i$ for which $i(K_X+B)\sim 0$.
For instance, if $(X,B)$ has index one, then $B$ is an integral divisor.
}
\end{definition}

\begin{definition}
\label{def:coreg}
{\em 
Let $(X,B)$ be a log Calabi--Yau pair.
The {\em coregularity} of $(X,B)$ is the dimension 
of a minimal log canonical center
in any dlt modification of $(X,B)$.
}
\end{definition}

The following lemma is well-known to the experts.
The proof follows from~\cite[Theorem 3]{dFKX} and~\cite[Lemma 3.2]{FMM22}.

\begin{lemma}\label{lem:cbir-pres}
Let $(X,B)$ and $(X',B')$ be two crepant birational equivalent log Calabi--Yau pairs.
Then, the pairs $(X,B)$ and $(X',B')$ have the same
index and the same coregularity.
\end{lemma} 

\begin{definition}\label{def:log-rat}
{\em
We say that a pair $(X,B)$ is {\em log rational} if there exists a crepant birational map
$\phi\colon (\pp^n,\Sigma^n) \dashrightarrow (X,B)$.
We say that a variety $X$ is {\em log rational type}
if there exists a boundary $B$ for which $(X,B)$ is a log rational pair.
}
\end{definition}

Note that a log rational pair must be log Calabi--Yau of index one and coregularity zero due to Lemma~\ref{lem:cbir-pres}. Furthermore, a variety of log rational type must be rational.

\begin{definition}\label{def:ct}
{\em 
We say that a pair $(X,B)$ is {\em of cluster type}
if there exists a crepant birational map
$\phi\colon (\pp^n,\Sigma^n)\dashrightarrow (X,B)$
such that ${\rm codim}_{\pp^n}({\rm Ex}(\phi)\cap \mathbb{G}_m^n)\geq 2$.
In other words, the crepant birational map
induces a birational contraction in the algebraic torus.
A variety $X$ is said to be {\em of cluster type} 
if there exists a boundary divisor $B$ on $X$ for which
the pair $(X,B)$ is of cluster type.
}
\end{definition}

It was proved by Corti \cite[Theorem 16.2.3]{Cor25}; and Enwright, the second author, and third author \cite[Theorem 1.3(3)]{EFM24} that Fano varieties of cluster type admit a dense open set covered by algebraic tori. More precisely, the following theorem holds:

\begin{theorem}\label{thm:Fano-ct}
Let $X$ be a $n$-dimensional $\qq$-factorial Fano variety and $(X,B)$ be a log Calabi--Yau pair. 
Then, the following statements are equivalent:
\begin{enumerate}
\item[(i)] the pair $(X,B)$ is of cluster type,
\item[(ii)] the open set $X\setminus B$ contains 
an algebraic torus of dimension $n$, and
\item[(iii)] the open set $X\setminus B$ is divisorially
covered by algebraic tori\footnote{This means that $X\setminus B$ is covered by algebraic tori up to a closed subset of codimension at least two.}.
\end{enumerate}
\end{theorem}

Now, we turn to recall some tools to prove that a given log Calabi--Yau pair is of cluster type.

\begin{lemma}\label{lem:cluster-type-under-birational}
Let $(X,B)$ be a log Calabi--Yau pair.
Let $\phi\colon Y\dashrightarrow X$ be a projective birational map. Assume that the divisorial locus of ${\rm Ex}(\phi)$ only consists of log canonical places $(X,B)$.
Let $(Y,B_Y)$ be the log pull-back of $(X,B)$ to $Y$.
If $(Y,B_Y)$ is of cluster type, then so $(X,B)$ is of cluster type.
\end{lemma} 

\begin{proof}
Assume that $(Y,B_Y)$ is of cluster type.
Then, there exists a crepant birational map 
$\psi\colon (\pp^n,\Sigma^n)\dashrightarrow (Y,B_Y)$ that induces a birational contraction in $\mathbb{G}_m^n$.
As $\phi$ only extracts divisors which are log canonical places of $(X,B)$, we conclude that the composition
$\phi\circ \psi \colon (\pp^n,\Sigma^n)\dashrightarrow (X,B)$
induces a birational contraction on the algebraic torus $\mathbb{G}_m^n$.
Therefore, the pair $(X,B)$ is of cluster type.
\end{proof} 

The following lemma gives a characterization
of cluster type pairs by means of dlt modifications
and the minimal model program.

\begin{lemma}\label{lem:MMP-charact-ct}
Let $(X,B)$ be a log Calabi--Yau pair of dimension $n$.
Then, the following conditions are equivalent:
\begin{enumerate}
\item the pair $(X,B)$ is of cluster type,
\item there exists a $\qq$-factorial dlt modification $(Y,B_Y)$ and an effective divisor $E_Y$ in $Y$ that contains no log canonical centers of $(Y,B_Y)$ such that for $\epsilon>0$ small enough the $(K_Y+B_Y+\epsilon E_Y)$-MMP terminates with a toric log Calabi--Yau pair after contracting all the components of $E_Y$.
\end{enumerate} 
\end{lemma}

\begin{proof}
First, assume that $(2)$ holds.
Let $\psi\colon Y\dashrightarrow T$ be the outcome of the $(K_Y+B_Y+\epsilon E_Y)$-MMP.
If we denote by $B_T$ the push-forward of $B_Y$ in $T$,
then the pair $(T,B_T)$ is a toric log Calabi--Yau pair.
Note that the crepant birational map 
$\psi^{-1}\colon (T,B_T)\dashrightarrow (Y,B_Y)$ only extracts canonical places of $(T,B_T)$, namely the components of $E_Y$, 
so it induces a birational contraction from
the algebraic torus $\mathbb{G}_m^n$.
On the other hand, as $g\colon (Y,B_Y)\rightarrow (X,B)$ is a dlt modification, then it only contracts components of $\lfloor B_Y\rfloor$, i.e. components with coefficient one.
Therefore, the composition
$g\circ \psi^{-1}\colon (T,B_T)\dashrightarrow (X,B)$
is a crepant birational map that induces a contraction
on the algebraic torurs $T\setminus B_T$.
Finally, pick any crepant toric birational map
$f\colon (\pp^n,\Sigma^n)\dashrightarrow (T,B_T)$.
As $f$ is toric, it is indeed an isomorphism on the torus $\mathbb{G}_m^n$.
We conclude that $ g\circ \psi^{-1} \circ f \colon (\pp^n,\Sigma^n)\dashrightarrow (X,B)$ is a crepant birational map 
that induces a birational contraction
on the algebraic torus $\mathbb{G}_m^n$.
Therefore, the pair $(X,B)$ is of cluster type.

The other direction follows from~\cite[Lemma 2.15 \& Lemma 2.29]{JM24}.
\end{proof}

\subsection{Pairs of coregularity zero and complexity one}\label{subsec:comp1}
In this subsection, we give criteria to be cluster type
for log Calabi--Yau pairs with sufficient boundary components.

\begin{definition}
{\em 
The {\em complexity} of a log Calabi--Yau pair 
$(X,B)$ is the following value:
\[
c(X,B):=\dim X + \dim_\qq {\rm Cl}_\qq(X) - |B|,
\]
where $|B|$ is the sum of the coefficients of $B$.
}
\end{definition}

The following theorem is proved by Brown, M\textsuperscript{c}Kernan, Svaldi, and Zong in~\cite[Theorem 1.2]{BMSZ18}.
It is known as the characterization of toric varieties via the complexity.

\begin{theorem}\label{thm:comp0}
Let $(X,B)$ be a log Calabi--Yau pair.
Then, we have $c(X,B)\geq 0$.
Furthermore, if $c(X,B)<1$, then the pair $(X,\lfloor B\rfloor)$ is a toric pair.\footnote{Here we mean that $X$ is toric and the boundary divisor $\lfloor B\rfloor$ is torus invariant.}
\end{theorem}

If the log Calabi--Yau pair has index one,
then the complexity is an integer.
In~\cite[Theorem 1.2]{Mor24b}, the author shows that a log Calabi--Yau pair
of index one, coregularity zero,
complexity one, and dimension at most three is log rational.
The proof therein actually shows that such a pair is of cluster type. 
Thus, the following statement holds. 

\begin{lemma}\label{lem:complexity-one}
Let $(X,B)$ be a log Calabi--Yau pair of index one, 
coregularity zero, and dimension at most three.
Assume that $c(X,B)=1$.
Then, the pair $(X,B)$ is of cluster type. 
\end{lemma} 

\subsection{Conic fibrations with two horizontal divisors}\label{subsec:2div}
In this subsection, we give criteria to be of cluster type for log Calabi--Yau $3$-folds admitting fibrations to surfaces.

\begin{definition}
{\em 
Let $f\colon X\rightarrow Y$ be a fibration.
Let $(X,B)$ be a log Calabi--Yau pair structure on $X$.
We say that $f\colon (X,B)\rightarrow (Y,B_Y)$ is a {\em crepant fibration} if $(Y,B_Y)$ is the log Calabi--Yau pair structure induced by the canonical bundle formula.
}
\end{definition}

Now, we turn to provide two lemmata that help us to detect whether log Calabi--Yau pairs admitting fibrations are of cluster type.

\begin{lemma}\label{lem:coreg-0-index-1-surface}
Let $(S,B_S)$ be a log Calabi--Yau surface of index one and coregularity zero.
Let $\pi\colon (S,B_S)\rightarrow (\pp^1,\{0\}+\{\infty\})$ be a crepant fibration.
Then, the pair $(S,B_S)$ is of cluster type.
\end{lemma} 

\begin{proof}
Applying a birational modification that only extracts
log canonical places of $(S,B_S)$, we obtain a new log Calabi--Yau surface $(S',B_{S'})$ over $\pp^1$
satisfying the following conditions:
\begin{itemize}
\item the variety $S'$ is $\qq$-factorial,
\item the morphism $\pi'\colon S'\rightarrow \pp^1$
has relative Picard rank one, and
\item the fibers ${\pi'}^{-1}(0)$ and ${\pi'}^{-1}(\infty)$ are contained in $\supp(B_{S'})$. 
\end{itemize} 
Then, we conclude that $c(S',B_{S'})\leq 1$.
Indeed, we have that $\rho(S')$, the Picard rank of $S'$,  is $2$, and $B_{S'}$ has at least three components, two verticals and at least one horizontal over $\pp^1$.
Therefore, by Lemma~\ref{lem:complexity-one}, we conclude that $(S',B_{S'})$ is of cluster type. 
Then, by Lemma~\ref{lem:cluster-type-under-birational}, we deduce that $(S,B_S)$ is of cluster type.
\end{proof}

\begin{lemma}\label{lem:two-horizontal-sections}
Let $(X,B)$ be a log Calabi--Yau $3$-fold of index one and coregularity zero.
Let $\pi\colon (X,B)\rightarrow (S,B_S)$ be a crepant fibration to a surface $S$. 
Assume that $B$ has two horizontal components over $S$
and that $(S,B_S)$ is of cluster type.
Then, the pair $(X,B)$ is of cluster type. 
\end{lemma} 

\begin{proof}
Since $(S,B_S)$ is of cluster type, there exists a crepant birational map
$(\pp^2,\Sigma^2)\dashrightarrow (S,B_S)$
that induces a birational contraction on $\mathbb{G}_m^2$.
By~\cite[Lemma 2.36]{Mor24a}, we have a commutative diagram
of crepant maps
\[
\xymatrix{
(X,B)\ar[d]_-{\pi} & (X',B')\ar@{-->}[l]_-{\phi}\ar[d]^-{\pi'} \\ 
(S,B_S) & (\pp^2,\Sigma^2)\ar@{-->}[l]_{\phi_S}
}
\]
where the following conditions are satisfied:
\begin{enumerate}
\item the variety $X'$ is $\qq$-factorial,
\item the crepant birational map $\phi$ only extracts log canonical places of $(X,B)$,
\item the crepant morphism $\pi'$ 
is a Mori fiber space,
\item the divisor $B'$ has two disjoint horizontal prime components over $\pp^2$, and
\item we have ${\pi'}^{-1}(\Sigma^2)\subseteq \supp(B')$.
\end{enumerate} 
Note that $B'$ has $5$ prime components; namely
the preimages of the prime components of $\Sigma^2$
and two horizontal components over $\pp^2$.
On the other hand, we know that $\rho(X')=2$.
Henceforth, the log Calabi--Yau pair $(X',B')$ is toric.
By Theorem~\ref{thm:comp0}
and Lemma~\ref{lem:cluster-type-under-birational},
we conclude that $(X,B)$ is of cluster type.
\end{proof}

\section{Relative cluster type}
\label{sec:rel-ct}

In this section, we introduce the concept of relative cluster type pairs and prove some preliminary lemmas.
The main statement of this section, roughly speaking, states that a log Calabi--Yau pair $(X,B)$ over a toric variety $T$ 
is of cluster type if and only if it is of relative cluster type over $T$. We will need some technical assumptions on the log canonical centers of $(X,B)$ being {\em compatible} with the toric strata of $(T,B_T)$ (see Theorem~\ref{thm:relative-cluster-type} below).

First, we give the definition of relative cluster type pair 
and state a remark and a lemma.

\begin{definition}\label{def:rel-ct}
{\em 
Let $(X,B)$ be a log Calabi--Yau pair 
and $X\rightarrow S$ be a projective contraction. 
We say that $(X,B)$ is {\em of relative cluster type over $S$}
if there exists a crepant birational contraction 
$(X,B)\dashrightarrow (T,B_T)$ over $S$ 
that only extracts log canonical places of $(X,B)$
and for which $(T,B_T)$ is a toric log Calabi--Yau pair. 
}
\end{definition}

\begin{remark}
{\em 
If we set $S={\rm Spec}(\cc)$ in Definition~\ref{def:rel-ct}, then we recover the usual definition of cluster type pair.
The previous statement follows from~\cite[Lemma 2.29]{JM24}.
On the other hand, if $(X,B)$ is of relative cluster type over $S$, then the following two conditions hold:
\begin{enumerate}
\item the variety $S$ is toric, and 
\item the pair $(S,B_S)$ induced by $(X,B)$ on $S$ via the canonical bundle formula is a toric log Calabi--Yau pair.
\end{enumerate} 
Condition (1) holds as $S$ is the image of a toric variety $T$ with respect to a projective contraction.
Condition (2) holds as the pair $(S,B_S)$ is also the pair induced by the canonical bundle formula for $(T,B_T)$ on $S$ with respect to the projective toric contraction $T\rightarrow S$.
}
\end{remark}

\begin{lemma}\label{lem:curves-on-toric}
Let $(T,B_T)$ be a projective toric log Calabi--Yau pair.
Let $C\subset T$ be a curve that is not contained in $B_T$.
Then, the curve $C$ intersects $B_T$ in at least two distinct points.
\end{lemma} 

\begin{proof}
As $T$ is projective, the boundary divisor $B_T$ supports an ample divisor
and hence $C$ must intersect $B_T$ in at least one point. 
For the sake of contradiction, we assume that $C$ intersects $B_T$ at a single point $t_0$.
Let $\Sigma$ be the fan associated to $T$, i.e.,
$T\simeq X(\Sigma)$.
The fan $\Sigma$ is a fan of polyhedral cones in the $\qq$-vector space $N_\qq$.
Let $V$ be the largest toric stratum of $T$ containing $t_0$.
Let $\sigma_V\in \Sigma$ be the cone associated to $V$.
Let $T'\rightarrow T$ be a small $\qq$-factorialization of $T$
corresponding to the fan refinement $\Sigma'\supseteq\Sigma$.
Let $(T',B_{T'})$ be the log pull-back of $(T,B_T)$ to $T'$.
Let $\sigma_{V,1},\dots,\sigma_{V,k}$ be the cones of $\Sigma'$ whose support are contained in the support of $\sigma_V$.
Let $V_1,\dots,V_k$ be the toric strata of $T'$ corresponding to the cones $\sigma_{V,1},\dots,\sigma_{V,k}$, respectively.
Let $C'$ be the strict transform of $C$ in $T'$.
Then $C'$ can only intersect $B_{T'}$ along the strata $V_1,\dots,V_k$
and it must intersect at least one of these strata.
Let $L$ be a linear form on $N_\qq$ that is strictly positive on the relative interior of $\sigma_V$.
In particular, the linear form $L$ is positive in
the relative interior of each of the cones $\sigma_{V,1},\dots,\sigma_{V,k}$.
By looking at the zeroes and poles of the rational form on $T'$ induced by $L$, a linear equivalence $D_1\sim D_2$ is induced between effective torus invariant divisors on $T'$ such that $D_1$ and $D_2$ have no common components.
Since $L$ is positive on the relative interior of each $\sigma_{V,1},\dots,\sigma_{V,k}$, we conclude that $D_1$ contains the strata $V_1,\dots,V_k$ while $D_2$ is disjoint from these strata.
Since $C'$ is not contained in $B_{T'}$, we get $D_1\cdot C'>0$ and so $D_2\cdot C'>0$, leading to a contradiction.
\end{proof} 

\begin{theorem}\label{thm:relative-cluster-type}
Let $(X,B)$ be a projective log Calabi--Yau pair.
Let $\phi\colon (X,B)\rightarrow (T,B_T)$ be a crepant contraction
to a projective toric log Calabi--Yau pair. 
Assume that $\phi^{-1}B_T \subseteq \supp(\lfloor B \rfloor)$. 
Then, the following conditions are equivalent:
\begin{enumerate}
\item the pair $(X,B)$ is of cluster type, and 
\item the pair $(X,B)$ is of cluster type over $T$.
\end{enumerate} 
\end{theorem}

\begin{proof}
From~\cite[Lemma 2.29]{JM24}, it follows that $(X,B)$ is of cluster type
if it is of cluster type over $T$.
Thus, it suffices to show that $(X,B)$ is of cluster type
over $T$ provided it is of cluster type.

Assume that $(X,B)$ is of cluster type.
By~\cite[Lemma 2.29]{JM24}, we know that there exists a $\qq$-factorial dlt modifcation $(Y,B_Y)\rightarrow (X,B)$ and a crepant birational contraction
$f\colon (Y,B_Y)\dashrightarrow (T_0,B_{T_0})$ to a toric log Calabi--Yau pair $(T_0,B_{T_0})$. 
Write $\psi \colon Y\rightarrow T$ for the corresponding projective contraction.
Let $E_Y:={\rm Ex}(f)\setminus B_Y$.
We may assume $(Y,B_Y+\epsilon E_Y)$ is dlt for $\epsilon>0$ small enough.
In particular, the divisor $E_Y$ contains no strata of the dlt pair $(Y,B_Y)$.
The $(K_Y+B_Y+\epsilon E_Y)$-MMP terminates in a toric model after contracting all the components of $E_Y$ (see, e.g.,~\cite[Lemma 2.15]{JM24}).
Thus, it suffices to show that every step of this MMP is relative over $T$.
In particular, it is enough to prove that every extremal curve contracted by this MMP is vertical over $T$.

Let $\pi\colon Y\rightarrow Y_1$ be the first contraction of this MMP, i.e., $\pi$ is either a divisorial contraction or a flipping contraction.
Let $R$ be the extremal ray in the cone of curves associated to the contraction $\pi$.
By assumption, the locus ${\rm Ex}(\pi)$ contains no lcc of $(Y,B_Y)$.
Indeed, we have ${\rm Ex}(\pi)\subseteq \supp(E_Y)$
and the ray $R$ must be $E_Y$-negative.
Let $W$ be a minimal\footnote{$W$ is minimal with respect to the inclusion.} lcc of $(Y,B_Y)$ among the log canonical centers of $(Y,B_Y)$
that contains a curve contracted by $Y\rightarrow Y_1$.
Note that $W$ exists as we consider $Y$ itself as a log canonical center of the pair $(Y,B_Y)$.
Furthermore, as ${\rm Ex}(\pi)$ contains no lcc of $(Y,B_Y)$, we know that $W$ has dimension at least two.
Let $(W,B_W)$ be the log Calabi--Yau pair induced by adjunction of $(Y,B_Y)$ to $W$.
Then, the induced morphism $\pi|_{W}\colon W\rightarrow W_1$ is a birational morphism and no curve contracted by $\pi|_{W}$ is contained in $B_W$.

The curves contained in $W$ and whose numerical class is contained in $R$
are contained in the exceptional locus of $\pi|_{W}$.
Let $C$ be an irreducible curve which is contained in $W$ 
and whose numerical class is contained in the ray $R$.
By construction, we know that $C$ is not contained in $B_W$.
Assume, by the sake of contradiction, that $C_T:=\psi(C)$ is a curve.
To get a contradiction, we will use the curve $C_T$ to violate the connectedness of log canonical centers. 
In the next paragraph, we analyze the image of $W$ on $T$.

Let $V:=\psi(W)$.
We argue that $V$ is a toric stratum and is the smallest toric stratum of $(T,B_T)$, including possibly $T$, that contains $C_T$.
The fact that $V$ is a toric strata follows as $\psi(W)$ is a log canonical center\footnote{Indeed, as $\psi$ is a crepant map then it maps log canonical centers to log canonical centers.} of $(T,B_T)$ and log canonical centers of a toric log Calabi--Yau pair are precisely the toric strata. 
If $V_0\subsetneq V$ is a toric stratum containing $C_T$, then we can write 
$V_0=B_{T,1}\cap \dots \cap B_{T,k}$, where the $B_{T,i}$ are prime components of $B_T$ and at least one $B_{T,i}$ does not contain $V$.
Let us say that $B_{T,1}$ does not contain $V$.
Therefore, it would follow that $C\subseteq \psi^{-1}(B_{T,1})\cap \dots \cap \psi^{-1}(B_{T,k})\cap W \subsetneq W$.
As $\psi^{-1}(B_{T,1})$ is contained in $\lfloor B_Y\rfloor$, we conclude that $C$ must be contained in a component of $\lfloor B_Y\rfloor$ that does not contain $W$. 
Therefore, the curve $C$ is contained in a lcc of $(Y,B_Y)$ which is properly contained in $W$. This is a contradiction.
Thus, we conclude that $V$ is the smallest toric stratum of $(T,B_T)$ that contains $C_T$. Let $(V,B_V)$ be the pair obtained by adjunction of $(T,B_T)$ to $V$. Then, the pair $(V,B_V)$ is a projective toric log Calabi--Yau pair and $C_T$ is a curve contained in $V$ which is not contained in $B_V$.
By Lemma~\ref{lem:curves-on-toric}, we conclude that $C_T$ intersects $B_V$ at two distinct points $p,q\in B_V$.

Now, we turn to contradict the fact that $C_T$ is a curve by invoking the connectedness of log canonical centers.
Let $B_{T,p}$ and $B_{T,q}$ be two prime components of $B_T$
such that $p\in B_{T,p}$ and $q\in B_{T,q}$ and 
neither $B_{T,p}$ nor $B_{T,q}$ contains $V$.\footnote{We do not assume that $B_{T,p}$ and $B_{T,q}$ are distinct prime components.}
Let $B_{1,p},\dots,B_{s_p,p}$ be the prime components of $\lfloor B_Y\rfloor$
that maps to $B_{T,p}$.
Analogously, let $B_{1,q},\dots,B_{s_q,q}$ be the prime components of $\lfloor B_Y\rfloor$ that maps to $B_{T,q}$.
By assumption, we know that $\sum_{i=1}^{s_p}B_{i,p}$
and $\sum_{i=1}^{s_q}B_{i,q}$ are the reduced fibers over $B_{T,p}$ and $B_{T,q}$, respectively.

Therefore, up to reordering, we may assume that $C$ intersects $B_{1,p}$
and $B_{1,q}$ in two distinct points $p_0$ and $q_0$.
Then, by construction, we know that $p_0,q_0$ are distinct points contained in $B_W$. Indeed, we have $B_{1,p}\cap W \subset B_W$ 
and $B_{1,q}\cap W \subset B_W$.
We recall that, by the minimality of $W$, we know that $C$ is not contained in $B_W$.
We conclude that $C$ is contained in a fiber of the birational morphism $\pi|_W\colon W\rightarrow W_1$ and
$C\cap {\rm Nonklt}(W,B_W)$ consists of at least two points $p_0$ and $q_0$. 
Let $W\rightarrow W_0\rightarrow W_1$ be the Stein factorization
of the morphism $W\rightarrow W_1$.
Let $\pi_{W,0}\colon W\rightarrow W_0$ 
be the induced contraction.
Since $W_0\rightarrow W_1$ is a finite morphism,
we conclude that $C$ must get mapped to a point $w_0$ in $W_0$.
Let $F:=\pi_{W,0}^{-1}(w_0)$ be the fiber over $w_0$.
Note that 
\[
F\cap {\rm Nonklt}(W,B_W) = F\cap \lfloor B_W\rfloor 
\]
contains no curves due to the minimality assumption of $W$.
Indeed, if $F\cap \lfloor B_W\rfloor$ contains a curve,
then $F'\cap \lfloor B_W\rfloor$ contains a curve, where
$F'$ is the fiber of $W\rightarrow W_1$ over the image of $w_0$ in $W_1$.
Hence, the set $F\cap {\rm Nonklt}(W,B_W)$ is just a finite set of closed points in $F$. 
As we have 
\[
C\cap {\rm Nonklt}(W,B_W) \subseteq F \cap {\rm Nonklt}(W,B_W)
\]
we conclude that the latter is a finite set of at least two closed points in $F$ and so it is a disconnected set.
Furthermore, the pair $(W,B_W)$ is log Calabi--Yau.
This contradicts~\cite[Theorem 1.1]{FS20}. 
We conclude that $\psi(C)$ is a point. 

From the previous paragraph, we know that the curve $C$ must be vertical over $T$. This implies that every curve in the ray $R$ is vertical over $T$.
Indeed, the pull-back to $Y$ of an ample divisor on $T$ intersects $C$ trivially and hence it must intersect any curve on $R$ trivially.

We just proved that the first step $(Y,B_Y+\epsilon E_Y)\dashrightarrow (Z,B_{Z}+\epsilon E_{Z})$ of the MMP occurs over $T$.
As usual, we denote by $B_Z$ and $E_Z$ the push-forward of $B_Y$
and $E_Y$ to $Z$, respectively.
Let $\psi_Z \colon Z\rightarrow T$ be the induced projective contraction.
Note that the following conditions are satisfied:
\begin{enumerate}
\item the pair on $T$ induced by the canonical bundle formula applied to $(Y,B_Y)$ and $(Z,B_Z)$ are the same,
\item we have $\psi_Z^{-1}(B_T)\subseteq \supp(\lfloor B_Z\rfloor)$, and 
\item the pair $(Z,B_Z+\epsilon E_Z)$ is dlt.
\end{enumerate} 
The first statement holds as $(Y,B_Y)$ and $(Z,B_Z)$ are crepant birational equivalent over $T$. 
The second statement holds as the birational transformation $Y\dashrightarrow Z$ over $T$ is $(\psi^*B_T)$-trivial so every extracted curve must be contained in $\supp(\lfloor B_Z\rfloor)$.
The last statement is a consequence of the fact that $Y\dashrightarrow Z$ is a step of the $(K_Y+B_Y+\epsilon E_Y)$-MMP.

We conclude that the new dlt log Calabi--Yau pair $(Z,B_{Z})$ satisfies the same initial hypotheses as $(Y,B_Y)$.
Thus, if we proceed inductively with the other steps of the MMP, we obtain a birational contraction $(Y,B_{Y})\dashrightarrow (T_0,B_{T_0})$ over $(T,B_T)$ such that $(T_0,B_{T_0})$ is a toric log Calabi--Yau pair. 
By Definition~\ref{def:rel-ct}, we conclude that $(X,B)$ is of cluster type over $(T,B_T)$. 
\end{proof}

\section{Criteria using toric blow-ups}
\label{sec:criteria-toric-blow-up}

In this section, we give a criteria for
(possibly singular) del Pezzo fibrations over $\pp^1$ to be cluster type. 
The main goal of this section is to prove Theorem~\ref{introthm:relative-criteria-irreducible}. 
This theorem will be used to study log Calabi--Yau pairs $(\pp^3,B)$ with $B$ reduced in the next section.
First, we start with the following lemma 
regarding weighted toric blow-ups
of log rational Calabi--Yau surfaces.

\begin{lemma}\label{lem:irreducible-nodal}
Let $(T,B_T)$ be a projective toric log Calabi--Yau surface.
Let $\pi\colon Y\rightarrow T$ be a projective birational morphism only extracting canonical places of $(T,B_T)$.
Assume that the finite set $\pi({\rm Ex}(\pi))$ contains no torus invariant points of $T$ and is contained in at most two prime components of $B_T$.
Let $(Y,B_Y)\rightarrow (X,B)$ be a crepant birational contraction for which $B$ is irreducible. 
Then, there is a toric weighted blow-up $\widetilde{X}\rightarrow X$ extracting a unique log canonical place of $(X,B)$ such that the strict transform $\widetilde{B}$ of $B$ in $\widetilde{X}$ has positive self-intersection.
\end{lemma} 

\begin{proof}
We note that $\pi({\rm Ex}(\pi))$ is non-empty.
Indeed, if $\pi({\rm Ex}(\pi))$ was empty, then 
$Y\rightarrow T$ would be an isomorphism and so 
$(X,B)$ would be a toric log Calabi--Yau pair.
This contradicts the fact that $B$ is irreducible.
Thus, we know that $\pi({\rm Ex}(\pi))$ is a non-empty finite set and is contained in at least one prime component of $B_T$.

Let $\Sigma$ be the fan of the projective toric surface $T$. 
Let $\rho_1,\rho_2\in \Sigma$ be the (possibly equal) rays corresponding to the prime components of $B_T$ containing $\pi({\rm Ex}(\pi))$.
We will proceed in two cases, depending on whether 
$\rho_1$ and $\rho_2$ span the same line in $\qq^2$.

First, assume that $\rho_1 = - \rho_2$.
Then, there is a toric fibration $T\rightarrow \pp^1$. 
We conclude that there is a fibration $\phi\colon X\rightarrow \pp^1$
and by construction 
$\phi^{-1}(0)\subset  \supp(B)$
and $\phi^{-1}(\infty)\subset \supp(B)$.
In particular, we conclude that the divisor
$B$ support two semiample divisors $B_1$ and $B_2$ with no common prime components.
Then, we conclude that the image of both $B_1$ and $B_2$ on $X$ must be divisors,
contradicting the fact that $B$ is irreducible. 

Now, assume that $\rho_1\neq - \rho_2$.
In particular, there is a linear function $L$ on $\qq^2$ which is strictly positive on the relative interior of both $\rho_1$ and $\rho_2$.
Let $\rho_0$ be the integral generator of the line in $\qq^2$ defined by $L$.
Let $\Sigma' \supset \Sigma$ be the fan refinement that introduces the rays $\rho_0$ and $-\rho_0$. 
Let $T'\rightarrow T$ be the induced toric blow-up.
As the center of the blow-up $T'\rightarrow T$ is disjoint from $\pi({\rm Ex}(\pi))$, then we get an induced blow-up $X'\rightarrow X$ making the obvious diagram commutative.
Let $(T',B_{T'})$ and $(X',B_{X'})$ be the log pull-back of $(T,B_T)$ to $T'$ and $X'$, respectively.
We have a toric fibration $T'\rightarrow \pp^1$
such that the exceptional divisors of $X'\rightarrow T'$ are mapped to the fiber over $\{\infty\}$.
Let $f'\colon X'\rightarrow \pp^1$ be the induced fibration.
The divisor $B_{X'}$ supports a semiample divisor whose image on $\pp^1$ via $f'$ is $\{0\}$.
We have a projective birational contraction $X'\rightarrow X$, 
so the strict transform of $B$
on $X'$ is contained in the fiber of $f'$ over $\{0\}$.
We denote by $B'$ the strict transform of $B$ on $X'$.
Let $X''$ be the surface obtained from $X'$ by contracting all the components of $B_{X'}$ 
contained in ${f'}^{-1}(0)$ and different from $B'$.
All these curves are contracted on $X$ so 
we have a projective birational contraction $X''\rightarrow X$.
Let $B_{X''}$ be the push-forward of $B_{X'}$ to $X''$. Let $B''$ be the strict transform of $B$ in $X''$.
Note we have a fibration $f''\colon X''\rightarrow \pp^1$ and $B''$ is a fiber of this fibration. 
Therefore, we have ${B''}^2=0$.
On the other hand, the projective birational contraction $X''\rightarrow X$
contracts at least two prime divisors of $B_{X''}$ that intersect $B''$, namely, the two horizontal components of $B_{X''}$ over $\pp^1$.
Thus, we may find a projective birational contraction $X''\rightarrow \widetilde{X}\rightarrow X$ that contracts all except one curves of $B_{X''}\setminus B''$.
If we let $\widetilde{B}$ to be the push-forward of $B''$ on $\widetilde{X}$, then we get $\widetilde{B}^2>0$. 
The projective birational morphism $\widetilde{X}\rightarrow X$ extracts a unique log canonical place of $(X,B)$ by construction.
Finally, the morphism $\widetilde{X}\rightarrow X$ is a toric weighted blow-up by~\cite[Theorem 1]{MS21}.
\end{proof} 

Now, we turn to give a criteria for possibly singular del Pezzo fibrations over $\pp^1$ to be of cluster type.

\begin{proof}[Proof of Theorem~\ref{introthm:relative-criteria-irreducible}]
First, assume that there exists a toric weighted blow-up $\widetilde{F}\rightarrow F$ extracting a unique log canonical place of $(F,B_F)$ such that the strict transform $\widetilde{B}_{F}$ of $B_F$ in $\widetilde{F}$ has positive self-intersection.
Let $C$ be the one-dimensional log canonical center of $(X,B)$ that maps generically one-to-one to $\pp^1$.
By performing a toric weighted blow-up at $C$ we obtain a projective crepant birational morphism $(\widetilde{X},B_{\widetilde{X}})\rightarrow (X,B)$
of relative Picard rank one 
such that the general log fiber of $(\widetilde{X},B_{\widetilde{X}})$ over $\pp^1$ is
$(\widetilde{F},B_{\widetilde{F}})$.
Let $E$ be the exceptional divisor of 
$\widetilde{X}\rightarrow X$.
Let $\widetilde{B}$ be the strict transform of $B$ on $\widetilde{X}$.
Note that the pair $(\widetilde{X},B_{\widetilde{X}})$ has complexity two. This follows as $\rho(\widetilde{X})=3$
and $B_{\widetilde{X}}$ has four prime components; 
$\widetilde{B}$ and $E$ that are horizontal over $\pp^1$ and two vertical components mapping to $\{0\}$ and $\{\infty\}$ in $\pp^1$.
We run the $(-E)$-MMP over $\pp^1$.
Since $\rho(\widetilde{X}/\pp^1)=2$, 
then there are two cases;
either the $(-E)$-MMP contracts a divisor 
and terminates with a Mori fiber space to $\pp^1$
or the $(-E)$-MMP consists of small modifications
followed by a Mori fiber space to a surface $S$ over $\pp^1$.
We analyze these cases in the following two paragraphs.

Firstly, assume that the $(-E)$-MMP consists of small modifications, a divisorial contraction, and terminates with a Mori fiber space to $\pp^1$.
Let $\widetilde{X}\dashrightarrow X'\rightarrow \pp^1$ be the outcome of this MMP.
Note that the small modifications over $\pp^1$ are isomorphisms on the log general fiber $(\widetilde{F},B_{\widetilde{F}})$
and the divisorial contraction 
induces a birational contraction on $\widetilde{F}$.
The birational contraction induced on $\widetilde{F}$ cannot contract $\widetilde{B}$ nor $E|_F$.
The first curve cannot be contracted as it has positive self-intersection while the second curve cannot be contracted as $E|_F^2$ is negative and we are running a $(-E)$-MMP.
Therefore, the image of both prime components $\widetilde{B}$ and $E$ on $X'$ are divisors.
Let $B_{X'}$ be the push-forward of $B_{\widetilde{X}}$ on $X'$.
Then, we conclude that $c(X',B_{X'})=1$.
Indeed, we have $\rho(X')=2$ and $B_{X'}$ has four components; two horizontal over $\pp^1$ and two vertical over $\pp^1$.
By Lemma~\ref{lem:complexity-one}, we conclude that $(X',B_{X'})$ is of cluster type.
By Lemma~\ref{lem:cluster-type-under-birational},
we deduce that $(X,B)$ is of cluster type.

Second, assume that the $(-E)$-MMP consists of small modifications and a Mori fiber space to a surface over $\pp^1$.
Let $\widetilde{X}\dashrightarrow X'\rightarrow S\rightarrow \pp^1$ be the outcome of this MMP.
Let $B_{X'}$ be the push-forward of $B_{\widetilde{X}}$ on $X'$.
Let $(S,B_S)$ be the log Calabi--Yau pair induced on $S$ by applying the canonical bundle formula to $(X',B_{X'})$.
We get an induced fibration 
$(\widetilde{F},B_{\widetilde{F}})\rightarrow \pp^1$ from the Mori fiber space
$(X',B_{X'})\rightarrow S$.
We argue that $B_{X'}$ has two prime components that are horizontal over $S$.
To show this, it suffices to argue that $B_{\widetilde{F}}$ has two components that are horizontal over $\pp^1$.
Indeed, we know that $B_{\widetilde{F}}$ has two components $\widetilde{B}_F$ and $E|_{\widetilde{F}}$.
The fibration $\widetilde{F}\rightarrow \pp^1$ is 
$E_{\widetilde{F}}$-positive so $E|_{\widetilde{F}}$ is horizontal over $\pp^1$.
On the other hand, as $\widetilde{B}_F^2>0$, the curve $\widetilde{B}_F$ is also horizontal over $\pp^1$.
We conclude that $B_{X'}$ has two prime components that are horizontal over $S$.
Henceforth, the log Calabi--Yau pair $(S,B_S)$ has coregularity zero and index one. 
Thus, from Lemma~\ref{lem:coreg-0-index-1-surface}, we conclude that $(S,B_S)$ is of cluster type
and then from Lemma~\ref{lem:two-horizontal-sections}, we conclude that $(X',B_{X'})$ is of cluster type.
Therefore, we deduce that $(X,B)$ is of cluster type
by Lemma~\ref{lem:cluster-type-under-birational}.

Now, we proceed to prove the second direction.
We assume that $(X,B)$ is of cluster type
and argue that the log general fiber $(F,B_F)$ admits the desired toric weighted blow-up.
Note that the conditions of Theorem~\ref{thm:relative-cluster-type} are satisfied, so 
the pair $(X,B)$ is of cluster type over $\pp^1$.
Hence, there exists a $\qq$-factorial dlt modification 
$(Y,B_Y)$ that admits a birational contraction 
to a toric log Calabi--Yau pair $(T,B_T)$ over $\pp^1$. Let $f\colon Y\dashrightarrow T$ be the birational contraction.
Let $E_Y:={\rm Ex}(f)\setminus B_Y$.
We can assume that $Y\dashrightarrow T$ is the outcome
of the $(K_Y+B_Y+\epsilon E_Y)$-MMP for $\epsilon>0$ small enough.
We argue that $E_Y$ has two prime components.
Since $T$ is toric and $\qq$-factorial, we know that 
\[
\rho(T/\pp^1)=|B_T|-4. 
\]
On the other hand, by assumption,  we know that 
$\rho(X/\pp^1)=1$ which can be rewritten as $\rho(X/\pp^1)=|B|-2$.
As $Y\rightarrow X$ is a $\qq$-factorial dlt modification, we conclude that 
$\rho(Y/\pp^1)=|B_Y|-2$. 
As we are assuming that $Y\dashrightarrow T$ only contracts divisorial canonical places of $(Y,B_Y)$, we conclude that $|B_Y|=|B_T|$.
We conclude that $E_Y$ consists of two prime components $E_{Y,1}$ and $E_{Y,2}$.
Let $B_{T,1}$ and $B_{T,2}$ be the (possibly equal) prime components of $B_T$ that contain the center of $E_{Y,1}$ and $E_{Y,2}$, respectively.
By replacing $T$ with a higher birational toric model, if necessary, we may assume that the center of $E_{Y,1}$ (resp. $E_{Y,2})$ contains no toric strata of $T$.

Let $(F_T,B_{F_T})$ be the log general fiber of $(T,B_T)\rightarrow \pp^1$. Note that 
$(F_T,B_{F_T})$ is a projective toric log Calabi--Yau surface.
Let $(F_Y,B_{F_Y})$ be the log general fiber of
$(Y,B_Y)$.
Then, we have a projective birational morphism
$\pi_F\colon F_Y \rightarrow F_T$ that only extract canonical places of $(F_T,B_{F_T})$.
Furthermore, we know that the finite set
$\pi_F({\rm Ex}(\pi_F))$ contains no torus invariant points of $F_T$ and is contained in at most two prime components of $B_{F_T}$; namely $B_{T,1}|_{F_T}$ and $B_{T,2}|_{F_T}$.
We have a crepant birational contraction
$(F_Y,B_{F_Y})\rightarrow (F,B_F)$
such that $B_F$ is an irreducible curve.
By Lemma~\ref{lem:irreducible-nodal}, we conclude that there exists a toric weighted blow-up
$\widetilde{F}\rightarrow F$ extracting a unique log canonical place of $(F,B_F)$ such that the strict transform of $B_F$ in $\widetilde{F}$ has positive self-intersection.
The nodal point of $B_F$ is the only
non-divisorial log canonical center of $(F,B_F)$.
Therefore, the toric weighted blow-up $\widetilde{F}\rightarrow F$ maps the exceptional curve to the nodal point of $B_F$.
\end{proof}

\section{Reducible quartic surfaces}
\label{sec:reducible}

In this section, we study coregularity zero log Calabi--Yau pairs of the form $(\pp^3,B)$ where $B$ is a reducible hypersurface.
We show that these pairs are of cluster type if and only if $B$ does not have the form $H+C$ where $H$ is a hyperplane, $C$ is a cubic hypersurface, and $H\cap C$ is a nodal cubic.

\begin{theorem}\label{thm:reducible}
Let $(\pp^3,B)$ be a log Calabi--Yau pair
of index one and coregularity zero.
Assume that $B$ is not irreducible.
Then, the pair $(\pp^3,B)$ is not of cluster type if and only if
$B=H+C$ where $H$ is a plane and $C$ is a cubic surface with no singularities through $H$.
\end{theorem} 

First, we prove the case in which $B$ has a point of multiplicity at least three.

\begin{proposition}\label{prop:mult3}
Let $(\pp^3,B)$ be a log Calabi--Yau pair of index one and coregularity zero. 
Assume there is a point $p\in \pp^3$ for which
${\rm mult}_p(B)\geq 3$.
Then, the pair $(\pp^3,B)$ is of cluster type.
\end{proposition}

\begin{proof}
Since the pair $(\pp^3,B)$ is log canonical, ${\rm mult}_p(B)\leq 3$. Therefore the multiplicity is exactly $3$, and $p$ is an lc center of $(\pp^3,B)$.
Let $\varphi: X\rightarrow \pp^3$ be the blow-up at $p$. The log pull-back of $(\pp^3,B)$ in $Y$ is $B'+E$, where $B'$ is the strict transform of $B$ via $\varphi$ and $E$ is the exceptional divisor of $\varphi$.
Consider the fibration $\pi:X \rightarrow E$ whose fibers are the strict transforms of lines through $p$. The fibration $\pi:(X,B'+E)\rightarrow (E,B\cap E )\simeq (\pp^2,B_1) $ is a crepant fibration, with $(\pp^2,B_1)$ of coregularity $0$, therefore of cluster type. 
Any component of $B'$ is horizontal over $\pp^2$, and so is $E$, therefore $\pi:(X,B'+E)\rightarrow (E,B'\cap E )\simeq (\pp^2,B_1) $  is a crepant fibration satisfying the hypothesis of Lemma~\ref{lem:two-horizontal-sections}. Therefore $(X,B'+E)$ is of cluster type. Hence $(\pp^3,B)$ is of cluster type by Lemma~\ref{lem:cluster-type-under-birational}.
\end{proof}

In particular, we get the following corollary.

\begin{corollary}\label{cor:comp3}
Let $(\pp^3,B)$ be a log Calabi--Yau pair of index one.
Assume that $B$ has at least three components.
Then, the pair $(\pp^3,B)$ is of cluster type.
\end{corollary}

\begin{proposition}\label{prop:two-quadrics}
Let $(\pp^3,B)$ be a log Calabi--Yau pair of index one and coregularity zero. Assume that $B=Q_1+Q_2$ two quadrics. Then $(\pp^3,B)$ is of cluster type. 
\end{proposition}

\begin{proof} 
Let $C$ be the intersection of $B_1$ and $B_2$. We will denote by $k$ the number of components of $C$.
Let $X \rightarrow \pp^3$ be the blow-up with center $C$. 

We will separate in cases according to the number of components of $C$.\\

\textit{Case 1:} $C$ has only one component.\\

Let $C_1$ be the preimage of the node of $C$ in $X$. Indeed, by adjunction, $(B_i,C)$ is a log Calabi--Yau pair of coregularity 0. 

Let $Y \rightarrow X$ be the blow-up of $C_1$. The log pull-back of $(\pp^3,B)$ at $Y$ is $(Y,B_{1,Y}+B_{2,Y}+E_Y+F)$, where $B_{1,Y}$ and $B_{2,Y}$ are the strict transforms of $B_1$ and $B_2$, while $E_Y$ is the strict transform of the exceptional divisor of the first blow-up, and $F$ is the exceptional divisor of the second blow-up. 

Consider the fibration $\pi:Y \rightarrow \pp^1$ whose fibers are the strict transform of the pencil formed by $B_1$ and $B_2$. 

We run a $(K_Y+B_{1,Y}+B_{2,Y}+E_Y)$-MMP over $\pp^1$. Let $Y \dashrightarrow Y' \rightarrow Z \rightarrow \pp^1$ be the outcome of this MMP, and $B_1',B_2',E'$ and $F'$ the strict transforms of $B_{1,Y},B_{2,Y},E_Y$ and $F$, respectively. Note that $Y' \rightarrow Z$ is a Mori fibered space.

If $Z$ is a curve, then $Z\simeq \pp^1$ and therefore $\rho(Y')=2$. This MMP does not contract $B_1$, $B_2$, $E$ or $F$. Thus $c(Y',B_1'+B_2'+E'+F')=1$. Therefore, by Lemma~\ref{lem:complexity-one}, the pair $(Y',B_1'+B_2'+E'+F')$ is of cluster type. 

If $Z$ is a surface, the pair $(Z,B_Z)$ induced by the canonical bundle formula satisfies the hypothesis of Lemma~\ref{lem:coreg-0-index-1-surface}, thus it is a pair of cluster type. The components $E'$ and $F'$ are horizontal over $Z$. Therefore by Lemma~\ref{lem:two-horizontal-sections}, the pair $(Y',B_1'+B_2'+E'+F')$ is of cluster type.\\

\textit{Case 2:} $C$ has two or more components.\\

Let $\hat{X} \rightarrow X$ be a small $\qq$-factorialization. The log pull-back of $(\pp^3,B_1+B_2)$ to $\hat{X}$ is $(\hat{X},\hat{B_1}+\hat{B_2}+E_1+\ldots+E_k)$, where $\hat{B_i}$ are the strict transforms of $B_i$ and $E_i$ are the exceptional divisors of the blow-up. Note that $k$ is the number of nodes of $C$.
Consider the fibration $\pi:\hat{X}\rightarrow \pp^1$ whose fibers are the strict transform of the pencil formed by $B_1$ and $B_2$. 

We run a $(K_{\hat{X}}+\hat{B_1}+\hat{B_2})$-MMP over $\pp^1$. Let $\hat{X} \dashrightarrow X' \rightarrow Z \rightarrow \pp^1$ be the outcome of this MMP, and $B_1',B_2'$ and $E_i'$ the strict transforms of $\hat{B_1},\hat{B_2}$ and $E_i$, respectively. Note that $X' \rightarrow Z$ is a Mori fibered space.

If $Z$ is a curve, then $Z\simeq \pp^1$, therefore $\rho(X')=2$. This MMP does not contract $\hat{B_1},\hat{B_2}$, $E_1$, \ldots, $E_k$. Thus $c(X',B_1'+B_2'+E_1'+\ldots + E_k')\leq 1$. Therefore, by Lemma~\ref{lem:complexity-one}, the pair $(X',B_1'+B_2'+E_1'+\ldots + E_k')$ is of cluster type.  

If $Z$ is a surface, the pair $(Z,B_Z)$ induced by the canonical bundle formula satisfies the hypothesis of Lemma~\ref{lem:coreg-0-index-1-surface}, thus it is a pair of cluster type. The components $E_1, \ldots, E_k$ are horizontal over $Z$. Therefore by Lemma~\ref{lem:two-horizontal-sections}, the pair $(X',B_1'+B_2'+E_1'+\ldots + E_k')$ is of cluster type.

In either case, by Lemma~\ref{lem:cluster-type-under-birational} the pair $(X,B_1+B_2)$ is of cluster type.
\end{proof}

For the remainder cases we will need a special case, similar to that of Theorem~\ref{introthm:relative-criteria-irreducible}. We will start by proving a combinatorial lemma that serves the role of Lemma~\ref{lem:irreducible-nodal} in the case of higher relative Picard rank. 

\begin{lemma}\label{lem:combinatorial-cycles}
Let $B=C_1+C_2+C_3$ be a cycle of curves in a surface $S$, all of them with negative self-intersection. Let $\mathcal{L}_B$ be the set of all the possible cycles of curves, obtained after performing blow-ups at nodes of cycles (up to symmetry) starting with $B$. Let $\mathcal{R}$ be the set of all the possible cycles of curves, obtained after performing blow-ups at nodes of cycles (up to symmetry) starting with the toric boundary of any Hirzebruch surface. 

For any $C_L=C_{L,1}+\ldots + C_{L,k}\in \mathcal{L}_B$ and $C_R=C_{R,1}+\ldots + C_{R,k} \in \mathcal{R}$, of the same length, there are at least three pairs $C_{L,i},C_{R,i}$ with different self-intersections, i.e. $C_{L,j}^2\neq C_{R,j}^2$. 
\end{lemma}

\begin{proof}
First of all, since blow-ups decrease the self-intersections of curves, any curve in a cycle in $\mathcal{L}_B$ must have negative self-intersection.

Assume by contradiction that there exist $C_L \in \mathcal{L}_B$ and $C_R \in \mathcal{R}$ having at most two curves with different self-intersections. We can choose them to be of minimum length. If there was a $(-1)$-curve in common, then by contracting it we obtain either cycles of lower length belonging to $ \mathcal{L}_B$ and $\mathcal{R}$, respectively or cycles formed by two curves. The first case would contradict minimality, while the second case would imply that there is a toric boundary on a surface with only two curves, a contradiction.

We study the possible number of $(-1)$-curves in $C_R$. By the previous parragraph, any of these $(-1)$-curves gives one pair of curves in $C_L$ and $C_R$ with different self-intersection, hence there must be at most $2$ $(-1)$-curves in $C_R$.\\

\noindent\textit{Case 1:} There are exactly $2$ $(-1)$-curves in $C_R$.\\

Then, $C_L$ cannot have any $(-1)$-curve, so it must be the cycle $B$ formed of three curves. As the two $(-1)$-curves in $C_R$ are the different curves, the third curve in $C_R$ must be the same as in $C_L$. Therefore $C_R$ has to be three curves with negative self-intersection, a contradiction.\\

\noindent\textit{Case 2:} There is exactly $1$ $(-1)$-curve in $C_R$.\\

This means that all the blow-ups performed to obtain $C_R$ are at points infinitely nearby in the minimal model, therefore $C_R$ must have at least $2$ curves with non-negative self-intersection . Since all curves in $C_L$ have negative self-intersection, there are at least three curves with different self-intersection in $C_L$ and $C_R$, namely, in $C_R$ the two curves with non-negative self-intersection and the one with self-intersection $-1$. A contradiction.\\

\noindent\textit{Case 3:} There is no $(-1)$-curve in $C_R$.\\

Therefore $C_R$ is a toric boundary, then three curves of which have non-negative self-intersections. Therefore there are three curves, giving pairs with different self-intersections in $C_L$ and $C_R$, a contradiction.

As in any case we obtain a contradiction, we achieve the desired result

\end{proof}

\begin{lemma}\label{lem:three-negative-curves-not-cluster-type}
Let $(X,B)$ be a $\qq$-factorial log Calabi--Yau $3$-fold of index one and coregularity zero.
Let $\pi\colon (X,B)\rightarrow (\pp^1,\{0\}+\{\infty\})$ be a crepant fibration of relative Picard rank three.
Assume that $\pi^{-1}(0)\cup \pi^{-1}(\infty)\subseteq \supp(B)$.
Let $(F,B_F)$ be a general log fiber of $\pi$
and assume that $B_F$ consists of a cycle of three curves. 

If all components of $B_F$ have negative self intersection, then $(X,B)$ is not of cluster type.
\end{lemma} 

\begin{proof}

Suppose by contradiction that $(X,B)$ is of cluster type. Note that the conditions of Theorem~\ref{thm:relative-cluster-type} are satisfied, so 
the pair $(X,B)$ is of cluster type over $\pp^1$.
Hence, there exists a $\qq$-factorial dlt modification 
$(Y,B_Y)$ that admits a birational contraction 
to a toric log Calabi--Yau pair $(T,B_T)$ over $\pp^1$. Let $f\colon Y\dashrightarrow T$ be the birational contraction.
Let $E_Y:={\rm Ex}(f)\setminus B_Y$.
By Lemma \ref{lem:MMP-charact-ct}, we can assume that $Y\dashrightarrow T$ is the outcome
of the $(K_Y+B_Y+\epsilon E_Y)$-MMP for $\epsilon>0$ small enough.
We will prove first that $E_Y$ has at most two prime components. Notice that $|B|\geq 5$.
Since $T$ is toric and $\qq$-factorial, we know that 
\[
\rho(T/\pp^1)=|B_T|-4. 
\]

On the other hand, by assumption,  we know that 
$\rho(X/\pp^1)=3$ which can be rewritten as $\rho(X/\pp^1)\leq |B|-2$.
As $Y\rightarrow X$ is a $\qq$-factorial dlt modification, we conclude that 
$\rho(Y/\pp^1) \leq |B_Y|-2$. 
As we are assuming that $Y\dashrightarrow T$ only contracts divisorial canonical places of $(Y,B_Y)$, we conclude that $|B_Y|=|B_T|$.
We conclude that $E_Y$ consists of at most two prime components $E_{Y,1}$ and $E_{Y,2}$.
Let $B_{T,1}$ and $B_{T,2}$ be the (possibly equal) prime components of $B_T$ that contain the center of $E_{Y,1}$ and $E_{Y,2}$, respectively.
By replacing $T$ with a higher birational toric model, if necessary, we may assume that the center of $E_{Y,1}$ (resp. $E_{Y,2})$ contains no toric strata of $T$. 

Let $(F_T,B_{F_T})$ and $(F_Y,B_{F_Y})$ be the log general fibers of $(T,B_T)\rightarrow \pp^1$ and $(Y,B_Y)\rightarrow \pp^1$, respectively. Then, we have a projective birational morphism $\pi_F:F_Y \rightarrow F_T$ that only extracts canonical places of $(F_T,B_{F_T})$. Furthermore, we know that the finite set $\pi_F({\rm Ex}(\pi_F))$ contains no torus invariant points of $F_T$ and is contained in at most two prime components of $B_{F_T}$. We have a crepant birational contraction $(F_Y,B_{F_Y})\rightarrow (F,B_F)$ such that $B_F$ is a cycle of three curves, each with negative self-intersection. 
Notice that $B_{F_T}$ is a complete toric boundary of a toric surface, hence obtained by blowing-up the boundary of a Hirzebruch surface at nodes, and $B_{F_Y}$ is obtained by blowing-up a cycle of three curves with negative self-intersection.  $B_{F_T}$ and $B_{F_Y}$ have different self-intersections at most two different curves, a contradiction by Lemma~\ref{lem:combinatorial-cycles} 

\end{proof}

\begin{proposition}\label{prop:two-components-cubic-hyperplane}
Let $(\pp^3,B)$ be a log Calabi--Yau pair of index one and coregularity zero. Assume that $B$ consists of a plane $H$ and a cubic surface $C$ with no singularities through $H$. Then $(\pp^3,B)$ is not of cluster type.
\end{proposition}

\begin{proof}
Let the intersection of $C$ and $H$ be called $\Gamma$. Let $m$ be the number of components of $\Gamma$. Since $(\pp^3,B)$ is log canonical of coregularity $0$, $\Gamma$ has $m$ nodal singularities.

Let $X\rightarrow \pp^3$ be the blow-up with center $\Gamma$. The log pull-back of $(\pp^3,B)$ to $X$ is $(X,H_X+C_X+E)$, where $H_X$ and $C_X$ are the strict transforms of $H$ and $C$, while $E=E_1+\ldots+E_m$ are the exceptional divisors. Since $\Gamma$ has $m$ nodal singularities, $X$ has $m$ nodal singular points inside $E$. Let $i(X)$ be the number of these singularities intersecting $C_X$.

Let $X'$ be a small $\qq$-factorialization of $X$. We have $\rho(X')=1+m$. By pulling-back $3H\sim C$, we obtain that $3H_{X'}+3E_{X'}\sim C_{X'}+E_{X'}$, therefore $C_{X'}\sim 3H_{X'}+2E_{X'}$. The intersection of $E_{X'}$ and $C_{X'}$ consists of $m+i(X)$ lines, with $m+i(X)$ nodes. For notation purposes, we define $m(Y)=m+i(X)$.

Let $Y\rightarrow X'$ be the blow-up with center $C_{X'}\cap E_{X'}$. The log pull-back of $(\pp^3,B)$ to $Y$ is $(Y,B_Y+C_Y+E_Y+F)$, where $F=F_{1}+\ldots+F_{m(Y)}$ are the exceptional divisors of $Y\rightarrow X'$. Since $C_{X'}\cap E_{X'}$ has $m(Y)$ nodes, $Y$ has $m(Y)$ nodal singularities, all of them in $F$. Let $i(Y)$ be the number of these singularities intersecting $C_Y$.

Let $Y'$ be a small $\qq$-factorialization of $Y$. We have that $\rho(Y')=\rho(X')+m(Y)=1+m+m(Y)$. By pulling-back $3H\sim C$, we obtain that $C_{Y'}\sim 3H_{Y'}+2E_{Y'}+F_{Y'}$. The intersection  $F_{Y'}\cap C_{Y'}$ consists of $m(Y)+i(Y)=m+i(X)+i(Y)$ many lines and the same amount of nodes.  For notation purposes, we define $m(Z)=m(Y)+i(Y)$.

Let $Z\rightarrow Y'$ be the blow-up with center $C_{Y'}\cap F_{Y'}$. The log pull-back of $(\pp^3,B)$ to $Z$ is $(Z,H_Z+C_Z+E_Z+F_Z+G)$, where $G=G_1+\ldots+G_{m(Z)}$ are the exceptional divisors of $Z \rightarrow Y'$.

Let $Z'$ be a small $\qq$-factorialization of $Z$. We have that $\rho(Z')=\rho(Y')+m(Z)=1+m+m(Y)+m(Z)$. By pulling-back $3H\sim C$, we obtain $C_{Z}\sim 3H_{Z'}+2E_{Z'}+F_{Z'}$. These are two disjoint sections of a linear system, thus we obtain a morphism to $\pp^1$, whose fibers are the strict transforms of the pencil formed by $3H$ and $C$. The general fiber is a cubic surface blown up at $i(X)+i(Y)$ points.

The fiber containing $H$ has other $m+m(Y)$ components that can be contracted. The divisors 
\[
G_{1,Z'},\ldots, G_{m(Z),Z'}
\]
are horizontal in this morphism, and cut divisors in the general fiber, $i(X)+i(Y)$ of them correspond to the exceptional divisors of the blow-ups performed to each of the cubic surfaces in the pencil in $\pp^3$. Therefore these $i(X)+i(Y)$ divisors get contracted, and $m(Z)-i(X)-i(Y)=m$ of these horizontal divisors remain. We end up with a morphism $(Z_1,H_{Z_1}+C_{Z_1}+G_{1,1}+\ldots+G_{1,m})\rightarrow (\pp^1, \{0\}+\{\infty\})$ of relative Picard rank $m$. We can blow-up $m-3$ times with center the horizontal sections formed by the nodes of $G_{1,1}+\ldots+G_{1,m}$ or its pullback.

We end up with a morphism $(Z_2,H_{Z_2}+C_{Z_2}+G'_1+G'_2+G'_3)\rightarrow (\pp^1, \{0\}+\{\infty\})$ of relative Picard rank $3$. Let $(F,B_F)$ be the general log fiber, we obtain that $B_F$ is formed by three curves. Now we compute their self-intersections in $F$, depending on $m$ the number of components of $\Gamma$.\\

\noindent\textit{Case 1:} If $\Gamma$ consists of three lines.\\

The self-intersection of each of these lines is $-1$ in a general cubic surface in the pencil formed by $H$ and $C$. In each of these cubic surfaces, the blown-up points were blown down in the opposite order, thus the self-intersections remain the same at the end.\\

\noindent\textit{Case 2:} If $\Gamma$ consists of a line and a conic.\\

The self-intersection of the line is $-1$, while the self-intersection of the conic is $0$ in a general cubic surface in the pencil formed by $H$ and $C$. In each of these cubic surfaces, the blown-up points were blown down in the opposite order, thus the self-intersections remain the same, before the last blow-up. After the last blow-up, the self-intersections drop and we obtain a cycle of three curves with self-intersections $-2$, $-1$ and $-1$.\\

\noindent\textit{Case 3:} If $\Gamma$ consists of a nodal cubic.\\

The self-intersection of the nodal cubic is $3$ in a general cubic surface in the pencil formed by $H$ and $C$. In each of these cubic surfaces the blown-up points were blown down in the opposite order, thus the self-intersections remain the same, before the last 2 blow-ups. After the second to last blow-up the self-intersections drop and we obtain a cycle of two curves with self-intersections $-1$ and $-1$. After the last blow-up, the three curves have self-intersection $-2$, $-2$ and $-1$.

In any case all the curves in $B_F$ have negative self-intersections. By Lemma~\ref{lem:three-negative-curves-not-cluster-type} the pair $(Z_2,H_{Z_2}+C_{Z_2}+G'_1+G'_2+G'_3)$ is not of cluster type. Since all the morphism performed extract only log canonical places, Lemma~\ref{lem:cluster-type-under-birational} implies that $(\pp^3,B)$ is not of cluster type.

\end{proof}

\begin{proof}[Proof of Theorem~\ref{thm:reducible}]
By Corollary~\ref{cor:comp3}, if $B$ has $3$ or more components then $(\pp^3,B)$ is of cluster type.
In the case of exactly two components, there are two possibilities; two quadrics or a hyperplane and a cubic. In the first case, the pair $(\pp^3,B)$ is of cluster type by Proposition~\ref{prop:two-quadrics}. In the other case, Proposition~\ref{prop:two-components-cubic-hyperplane} implies that the pair $(\pp^3,B)$ is not of cluster type when the cubic surface has no singularities through the plane, otherwise there will be some point of multiplicity $3$ in the intersection of the plane and the cubic surface, thus we can conclude by Proposition~\ref{prop:mult3}
\end{proof}

\section{Irreducible quartic hypersurfaces}
\label{sec:irred}

In this section, we study irreducible non-normal quartic hypersurfaces
$B\subset \pp^3$ of coregularity zero.
The non-normal assumption implies that 
$B$ has nodal singularities along a curve in $\pp^3$ (see, e.g.,~\cite[Proposition 3.1]{Duc24}). 
There are four cases depending on the locus of nodal points of $B$:
\begin{enumerate}
\item[(i)] the nodal locus of $B$ is the union of three concurrent lines in $\pp^3$; 
\item[(ii)] the nodal locus of $B$ is a twisted cubic in $\pp^3$;
\item[(iii)] the nodal locus of $B$ is a plane conic in $\pp^3$;
or
\item[(iv)] the nodal locus of $B$ is a line in $\pp^3$.
\end{enumerate} 
The previous follows from the classification of non-normal
quartic surfaces.
This classification dates back to Jessop~\cite{Jes16}.
We refer to Urabe's work for a more modern treatment~\cite{Ura86}.
Note that the nodal locus of $B$ agrees with the union
of the non-divisorial log canonical centers of the pair $(\pp^3,B)$. In each case, our argument to show that $(\pp^3,B)$ is of cluster type will use different features of the geometry of $B$. 

\begin{theorem}\label{thm:irreducible}
Let $(\pp^3,B)$ be a log Calabi--Yau pair 
of index one and coregularity zero.
Assume that $B$ is irreducible and non-normal.
Then, the pair $(\pp^3,B)$ is of cluster type unless
the nodal locus of $B$ is a plane curve in $\pp^3$.
\end{theorem} 

\begin{proof}
We proceed in two cases depending on the geometry of the log canonical centers of $(\pp^3,B)$.\\

\textit{Case 1:} Assume that the nodal locus of $B$ is the union of three concurrent lines in $\pp^3$.\\

Let $\ell_1\cup \ell_2 \cup \ell_3$ be the union of the $1$-dimensional log canonical centers of $(\pp^3,B)$.
Without loss of generality, we assume that the three lines 
$\ell_1,\ell_2$, and $\ell_3$ intersect at $[1:0:0:0]$.
Let $\pi\colon X\rightarrow \pp^3$ be the blow-up of $\pp^3$ at $\ell_1$.
Let $E$ be the exceptional divisor of $\pi$.
Then, we know that $E\simeq \pp^1\times \pp^1$.
The restriction of $\pi$ to $E$ is the projection to its first component.
Let $(X,\Delta)$ be the log pull-back of $(\pp^3,B)$ to $X$.
Then, we have $\Delta=B_X+E$, where $B_X$ is the strict transform of $B$ in $X$.
Consider the fibration $\phi\colon X\rightarrow \pp^1$ whose fibers are the strict transforms of planes containing $\ell_1$.
Note that $\ell_2$ and $\ell_3$ are contained in different planes $P_1$ and $P_2$ containing $\ell_1$.
Let $\ell_{2,X}$ and $\ell_{3,X}$ be the strict transforms of $\ell_2$ and $\ell_3$ on $X$.
Then, the intersections $x_1:=E\cap \ell_{2,X}$ and $x_2:=E\cap \ell_{3,X}$ are contained in different fibers of $\phi$.
In particular, the log Calabi--Yau pair induced on $\pp^1$ by the canonical bundle formula is $(\pp^1,\phi(x_1)+\phi(x_2))$.
Hence, the pair $(X,\Delta)$ has log canonical centers that map to $\phi_1(x_1)$ and $\phi(x_2)$, respectively.
Let $X'\rightarrow X$ be a projective birational morphism that extracts two prime divisors $E_1$ and $E_2$, over $\phi(x_1)$ and $\phi(x_2)$ respectively, which are log canonical places of $(X,\Delta)$.
We may assume that $X'$ is a $\qq$-factorial variety of Picard rank $4$.
Let $(X',\Delta')$ be the log pull-back of $(X,\Delta)$ to $X'$.
Then, we can write 
\[
\Delta'=B_{X'}+E_{X'}+E_1+E_2,
\]
where 
$B_{X'}$ (resp. $E_{X'})$ is the strict transform of $B$ (resp. $E$) in $X'$.
Note that $X'$ is a Fano type variety.
We run a $-(E_1+E_2)$-MMP over $\pp^1$ which terminates after contracting two divisorial canonical places of $(X',\Delta')$.
Let $X''\rightarrow \pp^1$ be the model where this MMP terminates.
Let $E_i''$ be the push-forward of $E_i$ to $X''$.
Let $B_{X''}$ and $E_{X''}$ be the push-forward of $B_{X'}$ and $E_{X'}$ to $X''$, respectively.
Then, the log Calabi--Yau pair
$(X'',B_{X''}+E_{X''}+E_1''+E_2'')$ has index one, 
coregularity zero, and complexity one.
By Lemma~\ref{lem:complexity-one}, we conclude that the aforementioned log Calabi--Yau pair is of cluster type.
Then, by Lemma~\ref{lem:cluster-type-under-birational}, we conclude that $(\pp^3,B)$ is of cluster type in this case.\\

\textit{Case 2:} Assume that the nodal locus of $B$ is a twisted cubic in $\pp^3$.\\ 

Let $\pi\colon X\rightarrow \pp^3$ be the blow-up of the twisted cubic. Let $E$ be the exceptional divisor
and $B_X$ be the strict transform of $B$ in $X$.
Write $\Delta=B_X+E$.
Therefore, the pair $(X,\Delta)$ is the log pull-back of $(\pp^3,B)$ to $X$.
Let $\phi\colon X\rightarrow \pp^2$ be the $\pp^1$-bundle induced by secant curves to the twisted cubic.
Then, the divisor $E$ maps to a smooth conic in $\pp^2$
and $B_X$ admits a $2$-to-$1$ surjective morphism to $\pp^2$.
Let $(\pp^2,\Gamma)$ be the log Calabi--Yau pair induced by the canonical bundle formula.
By~\cite[Theorem 6.1]{FFMP22}, we know that $(\pp^2,\Gamma)$ has index at most two and coregularity zero.
The divisor $\Gamma$ is a boundary divisor as $B_X$ maps $2$-to-$1$ to $\pp^1$.
Further, by the previous considerations, we know that $\Gamma$ contains a reduced conic.
Therefore, we can write $\Gamma=Q+\frac{1}{2}Q'$, where $Q'$ is a possibly degenerate conic. 
There are three cases depending on the geometry of $Q$:
\begin{itemize}
\item $Q'$ is a double line; 
\item $Q'$ is the sum of two lines intersecting transversally at $q\in Q$; or 
\item $Q'$ is a conic tangent to $Q$ at a point $q$.
\end{itemize}
We will prove that $(\pp^3,B)$ is of cluster type in each of the three cases.
The argument provided for the second case 
also applies to the first case.
Therefore, we will proceed in two cases.\\

\textit{Case 2.1:} Assume that $Q'$ is the sum of two lines intersecting transversally at $q\in Q$.\\

We blow-up $\pp^2$ at $q$
and then we inductively blow-up the intersection of the previous exceptional divisor with the strict transform of $Q$ 
until the strict transform of $Q$ has self-intersection zero.
Let $Y\rightarrow \pp^2$ be the induced projective birational morphism
with reduced exceptional divisor $E_Y$.
Let $(Y,Q_Y+\frac{1}{2}Q'_Y+E_Y)$ be the log pull-back of $(\pp^2,Q+\frac{1}{2}Q')$
to $Y$, where $Q_Y$ (resp. $Q'_Y$) is
the strict transform of $Q$ (resp. $Q'$) in $Y$.
Then, we have a fibration $Y\rightarrow \pp^1$ induced by the divisor $Q_Y$ with $Q_Y^2=0$.
Let $Z\rightarrow \pp^1$ be the outcome of the $K_Y$-MMP over $\pp^1$.
Let $Q_Z$ (resp. $Q'_Z$ and $E_Z$) be the push-forward of $Q_Y$ (resp. $Q'_Y$ and $E_Y$) to $Z$.
Then, the log Calabi--Yau pair
$(Z,Q_Z+\frac{1}{2}Q'_Z+E_Z)$ has index two and coregularity zero.
Furthermore, the aforementioned pair induces the log Calabi--Yau pair $(\pp^1,\{0\}+\{\infty\})$ on $\pp^1$ via the canonical bundle formula.
Furthermore, $Q'_{Y'}$ admits a two-to-one morphism to over $\pp^1$.
By~\cite[Lemma 2.11 \& Lemma 2.12]{MM24}, we get a commutative diagram of crepant maps:
\begin{equation}\label{eq:long-diag}
\xymatrix{ 
(\pp^3,B) & (X,B_X+E)\ar[l]\ar[d]_-{\phi} & (X',B_{X'}+E_{X'})\ar@{-->}[l]_-{f}\ar[d]_{\phi'} & (X'',B_{X''}+E_{X''})\ar@{-->}[l]_-{f'}\ar[d]_{\phi''} \ar[rd]^-{g} & \\
& (\pp^2,Q+\frac{1}{2}Q') & (Y,Q_Y+\frac{1}{2}Q'_Y + E_Y)\ar[l]\ar[r] & (Z,Q_Z+\frac{1}{2}Q'_Z+E_Z) \ar[r]^-{h} & \pp^1
}
\end{equation} 
where the following conditions are satisfied:
\begin{enumerate}
\item $f$ and $f'$ are crepant birational maps,
\item the crepant biratonal map $f\circ f'$ only extracts log canonical places of $(\pp^3,B)$, 
\item the morphisms $\phi,\phi'$, and $\phi''$ are
crepant fibrations of relative Picard rank one, and 
\item we have ${\phi'}^{-1}(Q_Y+E_Y)\subset \supp(B_{X'}+E_{X'})$
and 
${\phi''}^{-1}(Q_Z+E_Z)\subset \supp(B_{X''}+E_{X''})$.
\end{enumerate}
Let $(F,B_F+E_F)$ be a general log fiber of $g$,
where $B_F$ (resp. $E_F$) is the restriction
of $B_{X''}$ (resp. $E_{X''}$) to $F$.
The Mori fiber space $X''\rightarrow Z$
induces a fibration $F\rightarrow \pp^1$.
By construction, the morphism $\phi''$ is a conic bundle over $Z\setminus E_Z$.
In particular, the fibration 
$F\rightarrow \pp^1$ is a conic bundle
over the complement of the single point
$h^{-1}(p)\cap E_Z$.
As ${\phi''}^{-1}(E_Z)$ is contained in the support of $B_{X''}+E_{X''}$, we conclude that the only possibly reducible fiber of $(F,B_F+E_F)\rightarrow \pp^1$ is contained in the support of $E_F$.
Henceforth, the only singular points of $F$ are nodal points of $E_F+B_F$ and so every 
non-terminal place of $(F,B_F+E_F)$ is a log canonical place of $(F,B_F+E_F)$.
Let $W$ be the outcome of running a relative MMP over $Z$ for the canonical divisor
of a terminalization of $X''$.
Then, we have a Mori fiber space $W\rightarrow Z'$ to a higher birational model of $Z$.
Let $(W,B_W+E_W)$ be the log pull-back of $(\pp^3,B)$ to $W$, then the crepant birational map $(W,B_W+E_W)\dashrightarrow (\pp^3,B)$ only extract log canonical places of $(\pp^3,B)$.
Moreover, we have a Mori fiber space
$\phi_W \colon W\rightarrow Z'$ and a fibration $Z'\rightarrow \pp^1$.
Let $(F_W,B_{F_W}+E_{F_W})$ be a general log fiber of $(W,B_W+E_W)\rightarrow \pp^1$.
Here, as usual, $B_{F_W}$ (resp. $E_{F_W}$) is the restriction of $B_W$ (resp. $E_W$) to $F_W$.
The log Calabi--Yau pair $(F_W,B_{F_W}+E_{F_W})$ is obtained 
from $(F,B_F+E_F)$ by taking a minimal resolution
and then running a relative MMP over $\pp^1$ for the canonical divisor of the resolution.
The previous sentence implies that $F_W$ is smooth
and all the fibers of $F_W\rightarrow \pp^1$ are irreducible except possibly for the fiber supported in $E_{F_W}$ that could have two prime components.
Indeed, $B_{F_W}$ is ample over $\pp^1$ so every component of $E_{F_W}$ must intersect $B_{F_W}$ and so
$E_{F_W}$ has at most two prime components.
We conclude that there are four options for the isomorphism class of $F_W$; namely $\pp^1\times\pp^1$, $\Sigma_1$, or the blow-up at a point of one of the previous surfaces.
In any case, we have that $B_{F_W}^2>0$.
We run a $(-E_W)$-MMP over $\pp^1$.
Let $W\dashrightarrow W'$ be the induced minimal model program.
There are two cases; either this MMP eventually contracts a divisor
or it eventually induces a Mori fiber space to a surface.

Firstly, assume that $W\dashrightarrow W'$ contracts a divisor.
Then, the projective birational map $W\dashrightarrow W'$ must contract a divisor which is horizontal over $\pp^1$.
This divisor does not appear in the support of $B_W+E_W$.
Indeed, the induced fibration $F_W\rightarrow \pp^1$ is
$E_{F_W}$-positive and $E_{F_W}$ has negative self-intersection while $B_{F_W}$ has positive self-intersection.
Let $B_{W'}$ (resp. $E_{W'}$) be the push-forward of $B_W$ (resp. $E_W$) to $W'$.
The divisor $B_{W'}+E_{W'}$ has four prime components; two horizontal over $\pp^1$ and two vertical over $\pp^1$.
Furthermore, we have $\rho(W')=2$.
Then, the log Calabi--Yau pair $(W',B_{W'}+E_{W'})$ has complexity one.
Then, by Lemma~\ref{lem:complexity-one}, we know that 
$(W',B_{W'}+E_{W'})$ is of cluster type.
Thus, we conclude that $(\pp^3,B)$ is of cluster type
by Lemma~\ref{lem:cluster-type-under-birational}.

Secondly, assume that there is a Mori fiber space $W'\rightarrow S$ over $\pp^1$.
Let $B_{W'}$ (resp. $E_{W'}$) be the push-forward of $B_W$ (resp. $E_W$) in $W'$.
Then, the divisor $B_{W'}+E_{W'}$ has two prime components which are horizontal over $S$.
Indeed, the Mori fiber space $W'\rightarrow S$ over $\pp^1$ induces a $E_{F_W}$-positive fibration $F_W\rightarrow \pp^1$
and both $B_{F_W}$ and $E_{F_W}$ are horizontal over $\pp^1$.
Let $(S,B_S)$ be the log Calabi--Yau surface induced on $S$ by the canonical bundle formula for $(W',B_{W'}+E_{W'})$.
Then, the log Calabi--Yau surface $(S,B_S)$ has index one and coregularity zero.
Lemma~\ref{lem:coreg-0-index-1-surface} implies that $(S,B_S)$ is of cluster type.
By Lemma~\ref{lem:two-horizontal-sections}, we conclude that $(W',B_{W'}+E_{W'})$ is of cluster type.
Therefore, by Lemma~\ref{lem:cluster-type-under-birational}, we deduce that $(\pp^3,B)$ is of cluster type.\\

\textit{Case 2.2:} Assume that $Q'$ is a conic that is tangent to $Q$ at $q\in Q$.\\

First, we produce a sequence of birational transformations of $\pp^2$.
Let $Y_0\rightarrow \pp^2$ be the blow-up of $\pp^2$ at $q$. Let $E_0$ be the exceptional divisor.
Then, the pair $(Y_0,Q_{Y_0}+\frac{1}{2}Q'_{Y_0}+\frac{1}{2}E_0)$ is the log pull-back of $(\pp^2,Q+\frac{1}{2}Q')$ to $Y_0$, where $Q_{Y_0}$ (resp. $Q'_{Y_0}$)
is the strict transform of $Q$ (resp. $Q'$) in $Y_0$.
The three curves $Q_{Y_0}, Q'_{Y_0}$, and $E_0$ intersect transversally at a point $q_0$.
Let $Y_1\rightarrow Y_0$ be the blow-up of $Y_0$ at $q_0$.
We denote by $E_1$ the exceptional divisor of $Y_1\rightarrow Y_0$, by $E_0$ the strict transform of $E_0$ in $Y_1$, and by $Q_{Y_1}$ (resp. $Q'_{Y_1}$)
the strict transform of $Q_{Y_0}$ (resp. $Q'_{Y_0}$)
in $Y_1$.
Then, the log Calabi--Yau pair 
\[
\left(
Y_1,
Q_{Y_1}+\frac{1}{2}Q'_{Y_1}+\frac{1}{2}E_0+E_1
\right)
\]
is the log pull-back of $(\pp^2,Q+\frac{1}{2}Q')$ to $Y_1$.
Let $Y_1\rightarrow Y_2$ be the contraction of $E_0$ to an $A_1$-singularity.
Let $Q_{Y_2}$ (resp. $Q'_{Y_2}$ and $E_{Y_2}$)
be the push-forward of $Q_{Y_1}$  (resp. $Q'_{Y_1}$ and $E_1$), respectively.
Note that $Q_{Y_2}^2=2$.
Let $Y\rightarrow Y_2$ be the blow-up of $Q_{Y_2}\cap E_{Y_2}$ followed by the blow-up of the intersection of the strict transform of $Q_{Y_2}$ and the exceptional divisor of the first blow-up.
Let $(Y,Q_Y+\frac{1}{2}Q'_Y+E_Y)$ be the log pull-back of $(\pp^2,Q+\frac{1}{2}'Q)$ to $Y$ 
where $Q_Y$ (resp. $Q'_Y$) is the strict transform
of $Q$ (resp. $Q'$) in $Y$.
Then, we have $Q_Y^2=0$ and so we have an induced
fibration $Y\rightarrow \pp^1$.
We let $Z\rightarrow \pp^1$ be the outcome of the $K_Y$-MMP over $\pp^1$. 
Applying~\cite[Lemma 2.11 \& Lemma 2.12]{MM24}, we get a commutative diagram of crepant maps which is equal to~\eqref{eq:long-diag} with the only difference that this time the map $Y\dashrightarrow \pp^2$ is birational.
Nevertheless, it is still true in this case that the crepant birational map $f\circ f'$ in~\eqref{eq:long-diag} only extracts log canonical places of $(\pp^3,B)$.
Then, the rest of the argument to prove
that $(\pp^3,B)$ is of cluster type is verbatim from 
the last three paragraphs of Case 2.1.
\end{proof}

We turn to prove the main theorem of this article.

\begin{proof}[Proof of Theorem~\ref{introthm:cluster-type-(P^3,B)}]
If $B$ has at least three components, then the pair $(\pp^3,B)$ is of cluster type due to Corollary~\ref{cor:comp3}.
If $B$ has two components, then the statement of the theorem follows from Theorem~\ref{thm:reducible}.
If $B$ is irreducible, then the statement of the theorem follows from Theorem~\ref{thm:irreducible}
\end{proof} 

\section{Examples and questions}
\label{sec:ex-and-quest}

In this section, we provide some examples and some questions for further research.
We start by providing two examples in which Theorem~\ref{introthm:cluster-type-(P^3,B)} and
Proposition~\ref{prop:mult3} do not apply 
to detect if the log Calabi--Yau par
is of cluster type.

\begin{example}\label{ex:nodal-plane-conic}
{\em
In this example, we describe a log Calabi--Yau pair
$(\pp^3,B)$ satisfying the following conditions:
\begin{itemize}
\item $(\pp^3,B)$ has index one, 
\item $(\pp^3,B)$ has coregularity zero,
\item the unique one-dimensional log canonical center of $(\pp^3,B)$ is a plane conic, and 
\item $B$ has multiplicity at most two everywhere.
\end{itemize}
We consider the projective space $\pp^3$ with coordinates $[t:x:y:z]$.
Consider the quartic hypersurface $B$ in $\pp^3$ defined by
\[
\{[t:x:y:z]\in \mathbb{P}^3\mid (ty+x^2-z^2)^2-4x^2(x^2+y^2-z^2)=0.\}
\]
Set $\phi(x,y,z,t)=ty+x^2-z^2$ and 
$\psi(x,y,z,t)=x^2+y^2-z^2$. So, we can rewrite the equation of $B$ as 
\begin{equation}
\label{eq:short}
\phi^2-4x^2\psi=0.
\end{equation}
Thus, this hypersurface is smooth along $x\neq 0$.
So, its singular locus equals $\phi=x=0$
which is a smooth plane conic $C$.
Let $p\in C$ be a point different from 
$[1:0:0:0]$.
If $\psi(p)\neq 0$, then the equation~\eqref{eq:short}
gives us a simple nodal point.
If $\psi(p)=0$, then the equation~\eqref{eq:short}
gives us a pinch point.
In particular, $B$ has multiplicity at most two everywhere except possibly at $[1:0:0:0]$.
At the affine chart $t\neq 0$, we have 
\[
B\cap \mathbb{A}^3 =
\{
(x,y,z)\in \mathbb{A}^3 \mid 
(y+x^2-z^2)^2 - 4x^2(y^2+x^2-z^2)=0
\}. 
\]
Note that $B$ has multiplicity $2$ at the point $(0,0,0)$ of this chart.
Note that the equation defining $B\cap\mathbb{A}^3$ has multiplicity at least $4$ with respect to the weights $(1,2,1)$.
Performing the weighted blow-up 
with weights $(1,2,1)$, we obtain a projective birational morphism $Y\rightarrow \mathbb{A}^3$.
Let $(Y,B_Y+E)$ be the log pull-back of $(\mathbb{A}^3,B)$.
Then, the intersection $B_Y\cap E$
is given by 
\[
\{ [x:y:z]\in \pp(1,2,1) \mid (y+x^2-z^2)^2-4x^2(x^2-z^2)=0
\}.
\]
The previous equation defines a curve $C$ with a single node in $\pp(1,2,1)$. Thus, we conclude that $(Y,B_Y+E)$ has coregularity zero and so $(\pp^3,B)$ is a log Calabi--Yau pair of coregularity zero.
}
\end{example}

\begin{example}\label{ex:nodal-line}
{\em 
In this example, we describe a log Calabi--Yau pair $(\pp^3,B)$ satisfying the following conditions:
\begin{itemize}
\item $(\pp^3,B)$ has index one,
\item $(\pp^3,B)$ has coregularity zero,
\item the unique one-dimensional log canonical center of $(\pp^3,B)$ is a line, and 
\item $B$ has multiplicity at most two everywhere.
\end{itemize}
We consider the projective space $\pp^3$ with coordinates $[t:x:y:z]$.
Consider the quartic hypersurface in $\pp^3$ defined by
\[
B:=\{ [t:x:y:z]\in \pp^3 \mid t^2z^2 +  xyzt + x^2y^2 + x^4 +t^4 =0 \}.
\]
By the Jacobian criterion, we can see that the line $\ell = \{t = x = 0\}$ is the whole singular locus of $B$.
Further, by looking at the Hessian, we can see that $B$ has simple nodal singularities along $\ell$ outside the points $p:=[0:0:1:0]$ and $q:=[0:0:0:1]$.
Thus, it suffices to study $B$ along these two points.
At the chart $y\neq 0$, the quartic hypersurface becomes
\begin{equation}\label{eq:line}
B\cap \mathbb{A}^3 = \{ (t,x,z)\in \mathbb{A}^3 \mid t^2z^2+xzt+x^2+x^4+t^4=0\}.
\end{equation}
The previous equation has multiplicity two at 
$(0,0,0)$ and is a degenerate cusp singularity so $B$ is an slc surface.
Furthermore, the equation defining the affine chart~\eqref{eq:line} is homogeneous of weight four with respect to the weight $(1,2,1)$.
Therefore, $p$ is a log canonical center of $B$.
By symmetry, the same happens at $q$.
We conclude that the pair obtained by adjunction of $(\pp^3,B)$ to $\ell$ is $(\ell,p+q)$ and has coregularity zero. 
Let $Y\rightarrow \pp^3$ be the blow-up of $\ell$ with exceptional divisor $E$.
Let $B_Y$ be the strict transform of $B$ in $Y$.
Then, $B_Y\cap E$ is reducible over a neighborhood of $p$, so the pair $(Y,B_Y+E)$ has coregularity zero.
This implies that $(\pp^3,B)$ has coregularity zero.
}
\end{example}

\begin{example}\label{ex:irred-normal}
{\em 
In this example, we describe a log Calabi--Yau pair $(\pp^3,B)$ satisfying the following conditions:
\begin{itemize}
\item $(\pp^3,B)$ has index one,
\item $(\pp^3,B)$ has coregularity zero,
\item $B$ is an irreducible normal quartic surface with a single singular point, and
\item the unique zero-dimensional log canonical center of $(\pp^3,B)$ is a point.

\end{itemize}
We consider the projective space $\pp^3$ with coordinates $[t:x:y:z]$.
Consider the quartic hypersurface in $\pp^3$ defined by
\[
B:=\{ [t:x:y:z]\in \pp^3 \mid xyzt + x^4 + y^4 +z^4 =0 \}.
\]
By the Jacobian criterion, we can see that the single point $p:=[1:0:0:0]$ is the singular locus of $B$. Since $B$ is a complete intersection and regular in codimension 1, by \cite[Proposition II.2.23(a)]{Har77} we have that $B$ is normal.

The singularity at $p$ is not rational. Set $C \coloneqq \{xyz=0 \} \subset \pp^2$ and $C' \coloneqq \{x^4 + y^4 + z^4 =0\} \subset \pp^2$. By looking at the classification of normal quartic surfaces with irrational singularities established by Ishii \& Nakayama \cite{IN04}, we see that $B$ fits into Type D and its minimal resolution is a rational surface $M$ of Picard number 13 given by the blowup of $\pp^2$ at the base points of the pencil generated by $C$ and $C'$. Let $\rho$ be such blowup and $E$ be the corresponding exceptional divisor. The birational morphism $\rho \colon (M,E) \rightarrow (\pp
^2,C)$ is a crepant map between Calabi-Yau pairs and since $(\pp^2,C)$ has coregularity zero, then $(M,E)$ so does.

Set $D$ to be the strict transform of $C'$. By \cite[Proposition 1.4 \& Section 2.2]{IN04}, one has that $|D|$ is base point free and the corresponding morphism $\sigma$ is the minimal resolution of $S$. We have the following commutative diagram of crepant maps:
\begin{equation}
\xymatrix{ 
 & (M,E)\ar[ld]_-{\sigma} \ar[rd]^-{\rho} & \\
(S,0) \ar@{-->}[rr] &  & (\pp^2,C). 
}
\end{equation} 

Therefore, $(S,0)$ has coregularity zero and by \cite[Lemma 2.28]{FJJ22}, we deduce that $(\pp^3,B)$ has coregularity zero as well. Note that $p$ has multiplicity three and therefore by Proposition \ref{prop:mult3} the log Calabi-Yau pair $(\pp^3,B)$ is cluster type.
}
\end{example}

The first two examples motivate the first question regarding families of log Calabi--Yau pairs $(\pp^3,B)$ of coregularity zero.

\begin{question}\label{quest:number}
How many families of log Calabi--Yau pairs $(\pp^3,B)$ of coregularity zero with $B$ irreducible and nodal along a plane curve are there?
\end{question} 

We expect that there are at least three families, corresponding to the cases in which the nodal curve is a line, two concurrent lines, and a plane conic.
However, we do not know if each of these cases corresponds to a single family of pairs.
An answer to Question~\ref{quest:number} and a more explicit description of such families, 
would be a first step towards settling the classification
of cluster type log Calabi--Yau pairs of the form $(\pp^3,B)$.

In~\cite[Theorem 1.11]{EFM24}, the authors show that 
whenever $X$ is a Fano surface and $(X,B)$ is a cluster type log Calabi--Yau pair, the open set $X\setminus B$
is divisorially covered by at most two algebraic tori.
In the case of cluster type log Calabi--Yau pairs 
$(\pp^3,B)$, using similar arguments, one can show that 
$\pp^3\setminus B$ can be divisorially covered by at most three copies of $\mathbb{G}_m^3$.
If $\pp^3\setminus B$ consists of a single algebraic tori, then $B$ is the sum of four hyperplanes and the pair is toric (see, e.g.,~\cite[Theorem 1.10]{EFM24}).
Thus, any cluster type log Calabi--Yau pair $(\pp^3,B)$, which is not toric, satisfies that $\pp^3\setminus B$ needs at least two algebraic tori to be divisorially covered.
This motivates the following question.

\begin{question}
Which cluster type log Calabi--Yau pairs $(\pp^3,B)$ cannot be divisorially covered with two algebraic tori? 
\end{question}

It is unclear to us how the number of algebraic tori needed to divisorially cover $\pp^3\setminus B$ reflects on the geometry of $B$.

\bibliographystyle{habbvr}
\bibliography{bib}

\vspace{0.5cm}
\end{document}